\documentclass[11pt]{article}

\usepackage{amsmath,amsthm,verbatim,amssymb,amsfonts,amscd, graphicx,color, framed,enumerate}

\newcommand{\R}{\mathbb R}
\newcommand{\Z}{\mathbb Z}
\newcommand{\N}{\mathbb N}
\usepackage[left=1in,right=1in,top=1in,bottom=1in]{geometry}
\usepackage[colorlinks]{hyperref}
\usepackage{tikz}
\usetikzlibrary{positioning}
\usetikzlibrary{shapes}
\usepackage{multicol}

\theoremstyle{plain}
\newtheorem{theorem}{Theorem}[section]
\newtheorem{corollary}{Corollary}[section]
\newtheorem{lemma}{Lemma}[section]
\newtheorem{remark}{Remark}
\newtheorem{proposition}{Proposition}[section]

\theoremstyle{definition}
\newtheorem{definition}{Definition}[section]

\theoremstyle{definition}
\newtheorem{example}{Example}

\def\S{\mathcal S}
\def\C{\mathcal C}
\def\Re{\mathcal R}

\def\var{\text{Var}}
\def\P{\mathbb{P}}
\def\E{\mathbb{E}}
\def\P{\mathbb{P}}

\begin{document}

\title{Deficiency zero for random reaction networks under a  stochastic block model 
framework}
\author{David F.~Anderson\thanks{Department of Mathematics, University of
  Wisconsin, Madison, USA.  anderson@math.wisc.edu.},
  \and
  Tung D.~Nguyen\thanks{Department of Mathematics, University of
  Wisconsin, Madison, USA.  nguyen34@math.wisc.edu.
}
}
\maketitle

\begin{abstract}
Deficiency zero is an important network structure and has been the focus of many celebrated results within reaction network theory. In our previous paper \textit{Prevalence of deficiency zero reaction networks in an Erd\H os-R\'enyi framework}, we provided a framework to quantify the prevalence of deficiency zero  among randomly generated reaction networks. Specifically, given a randomly generated binary reaction network with $n$ species, with an edge between two arbitrary vertices occurring independently with probability $p_n$, we established the threshold function $r(n)=\frac{1}{n^3}$ such that the probability of the random network being deficiency zero converges to 1 if $\frac{p_n}{r(n)}\to 0$ and converges to 0 if $\frac{p_n}{r(n)}\to\infty$, as $n \to \infty$.

With the base Erd\H os-R\'enyi framework as a starting point, the current paper provides a significantly more flexible framework by  weighting the edge probabilities via control parameters $\alpha_{i,j}$, with $i,j\in \{0,1,2\}$ enumerating the types of possible vertices (zeroth, first, or second order). The control parameters can be chosen to generate random reaction networks with a specific underlying structure, such as ``closed'' networks with very few inflow and outflow reactions, or ``open'' networks with abundant inflow and outflow. Under this new framework, for each choice of control parameters $\{\alpha_{i,j}\}$,  we establish a threshold function $r(n,\{\alpha_{i,j}\})$ such that the probability of the random network being deficiency zero converges to 1 if $\frac{p_n}{r(n,\{\alpha_{i,j}\})}\to 0$ and converges to 0 if $\frac{p_n}{r(n,\{\alpha_{i,j}\})}\to \infty$.

\end{abstract}

\section{Introduction}
Reaction networks are used to model a variety of physical systems from microscopic processes such as chemical reactions and protein interactions, to macroscopic phenomena such as the spread of epidemic disease and the evolution of species. In reaction networks, the interacting agents (such as biochemical molecules, animal species, human populations) are referred to by a common term ``species". These networks take the form of directed graphs in which the vertices, often termed \textit{complexes} in the domains of interest, are linear combinations of the species over the non-negative integers and the directed edges, which imply a state transition for the associated dynamical system, are termed \textit{reactions}.  See Figure \ref{figure1} for an example of a reaction network.

To each such graph a quantity termed the \textit{deficiency} can be computed, and this quantity is central to many classical and celebrated results in the field \cite{ACKK:explosion, ACK:ACR, AC:non-mass, ACK:product, AN:non-mass,CW:CB,F1,H,H-J1}. To compute the deficiency, we first note that the vertices of a reaction network, which will be denoted by $y$ and/or $y'$ throughout this paper, can be viewed as vectors in $\Z^n_{\ge0}$.  For example, the vertices in Figure \ref{figure1}  can  be associated with the vectors $\left[\begin{matrix}0\\0\end{matrix}\right], \left[\begin{matrix}1\\1\end{matrix}\right], \left[\begin{matrix}0\\1\end{matrix}\right], \left[\begin{matrix}2\\0\end{matrix}\right], \left[\begin{matrix}0\\2\end{matrix}\right].$  Moreover, a directed edge between two such vectors, $y \to y'$, implies a state update of the form $y'-y \in \Z^n$.  The set of  state update vectors implied by the graph is called the set of  ``reaction vectors'' for the model.   Viewing things in this manner the deficiency, $\delta$, for the graph  is
\[
    \delta = \#\text{vertices} - \#\text{connected components} - \dim( \text{span}(\text{reaction vectors})).
\]
For example, the deficiency of the reaction network in Figure \ref{figure1} is $\delta = 5 - 2 - 2 = 1.$ Given the significance of deficiency zero, a natural question then arises: 

\vspace{.1in}
\noindent Question:  \textit{how prevalent are deficiency zero networks, and, in particular, are they more prevalent in certain natural settings than expected?}

\vspace{.1in}

\begin{figure}\label{figure1}
\begin{center}
\begin{tikzpicture}
    \node (0)   at (4,0)  {$\emptyset$};
    \node (S1+S2)   at (6.8,0)  {$S_1+S_2$};
    \node (S2)  at (6.8,-2.5) {$S_2$};

    \path[->]
    (0) edge  (S1+S2)
    (S1+S2) edge  (S2)
    (S2) edge (0);   
    
    \node (2S1) at (9,-1) {$2S_1$};
    \node (2S2) at (11,-1) {$2S_2$};
    
    \path[->]
    ([yshift=-1.4ex]2S1.north east) edge ([yshift=-1.4ex]2S2.north west)
    ([yshift=1.4ex]2S2.south west) edge ([yshift=1.4ex]2S1.south east);

\end{tikzpicture}
\end{center}
\caption{A reaction network with two species: $S_1$ and $S_2$. The vertices are linear combinations of the species over the integers.  The directed edges are termed reactions and determine the net change in the counts of the species due to one instance of the reaction.  For example, the reaction $S_1 + S_2 \to S_2$ reduces the count of $S_1$ by one, but does not affect the count of $S_2$.
}
\end{figure}
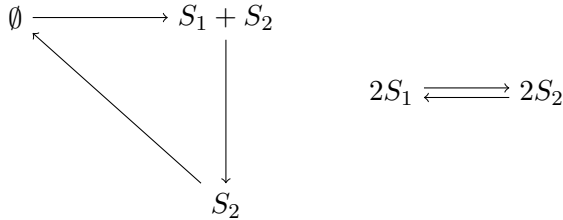

To begin to address this question, our previous paper \cite{prevalence} sought to formulate a framework for deciding the prevalence of deficiency zero among  reaction networks with large numbers of species and vertices. In particular, in \cite{prevalence} we considered random reaction networks generated by an  Erd\H os-R\'enyi random graph framework in the large species limit. We assumed a species set of $\{S_1,\dots,S_n\}$, and, because of their relevance in the biology and chemistry literature, focused on binary reaction networks (whose vertices are of the form $\emptyset$, $S_i$, or $S_i+S_j$, see the next section for an explanation of terminology). We then assumed  that the probability of  an edge, or reaction, between any two vertices, which we denoted by $p_n \in (0,1)$, was fixed.    We then derived a threshold function $r(n)=\frac{1}{n^3}$ such that the probability of the random binary reaction network being deficiency zero converges to $1$, as $n\to \infty$, if $\frac{p_n}{r(n)}\to 0$ and converges to $0$ if $\frac{p_n}{r(n)}\to \infty$. Here we use the usual notation that  for two sequences $\{a_n\}_{n=0}^{\infty} \text{ and } \{b_n\}_{n=0}^{\infty}$ in $\R$ we write $a_n\ll b_n$ or $b_n\gg a_n$ if $\lim_{n\to\infty}\frac{a_n}{b_n}=0$ and $a_n\sim b_n$ if $\lim_{n\to\infty}\frac{a_n}{b_n}=c$ for some constant $c$.

While the basic Erd\H os-R\'enyi framework can serve as a good starting point due to its simplicity, in practice one may want to use a more flexible framework that can be easily adapted to different settings where reaction networks may have different underlying structures. For example, one may want to study a closed system where inflow and outflow reactions such as $\emptyset\leftrightharpoons S_i$ are prohibited. On the other hand, one could be interested in an open system where inflow and outflow reactions are abundant. In another setting, perhaps one wants to only allow for reactions that preserve the number of molecules such as $S_i\leftrightharpoons S_j$ or $S_i+S_j\leftrightharpoons S_h+S_k$. 

To properly generate random reaction networks in those situations, the current paper considers a stochastic block model framework--a  generalized Erd\H os-R\'enyi framework with weighted edge probabilities \cite{blockmodel}. In particular, given that there are $n$ species, we partition the set of all possible reactions into different classes such as $E_n^{0,1}=\{\emptyset \leftrightharpoons S_i$\}, $E_n^{1,1}=\{S_i\leftrightharpoons S_j$\}, $E_n^{0,2}=\{\emptyset\leftrightharpoons S_i+S_j\}$, and so on. In the class with the highest amount of reactions $E_n^{2,2}=\{S_i+S_j\leftrightharpoons S_h+S_k$\}, the edge probability is $p_n$. For each relevant pair of $i$ and $j$ the edge probability for the reaction class $E^{i,j}_n$ is given by $n^{\alpha_{i,j}}p_n$, where $\{\alpha_{i,j}\}$ are parameters that can be used to control the structure of the random reaction networks. For example, $\alpha_{0,1}=\alpha_{0,2}=0$ could be used to model closed systems with very few inflow and outflow reactions.  Given each choice of $\{\alpha_{i,j}\}$, we will provide a threshold function $r(n,\{\alpha_{i,j}\})$ such that the probability of the random binary reaction network being deficiency zero converges to $1$ if $\frac{p_n}{r(n,\{\alpha_{i,j}\})}\to 0$, as $n\to \infty,$ and converges to $0$ if $\frac{p_n}{r(n,\{\alpha_{i,j}\})}\to \infty$. For the sake of brevity, we will write $r(n)$ instead of $r(n, \{\alpha_{i,j}\})$ throughout the rest of the work.

The remainder of this paper is organized as follows. In Section \ref{sec3}, we briefly review some key definitions of reaction network theory, including deficiency. In Section \ref{sec4}, we formally set up the stochastic block model framework briefly described above for generating random reaction networks, and provide some concrete examples with illustrations. In Section \ref{sec5}, we provide a set of conditions for deficiency zero in terms of the expected number of reactions in each reaction class. Finally in Section \ref{sec6}, we provide an algorithm to derive the threshold function $r(n)$ for a given choice of control parameters $\{\alpha_{i,j}\}$, which is based on the theoretical results derived in Section \ref{sec5}. 

\section{Reaction networks}\label{sec3}
\subsection{Reaction networks and key definitions}
Let $\{S_1,\dots,S_n\}$ be a set of $n$ \textit{species} undergoing a finite number of reaction types.  We denote a particular reaction by $y \to y'$, where $y$ and $y'$ are linear combinations of the species on $\{0,1,2,\dots\}$ representing the number of molecules of each species consumed and created in one instance of that reaction, respectively. The linear combinations $y$ and $y'$ are often called \textit{complexes} of the system. For a given reaction, $y\to y'$, the complex $y$ is called  the \textit{source complex} and $y'$ is called the \textit{product complex}.  A complex can be both a source complex and a product complex.  We may associate each complex with a vector in $\mathbb{Z}^n_{\geq 0}$, whose coordinates give the number of molecules of the corresponding species in the complex. As is common in the reaction network literature, both ways of representing complexes will be used interchangeably throughout the paper. For example,  if the system has 2 species $\{S_1,S_2\}$, the reaction $S_1+S_2\to 2S_2$ has $y=S_1+S_2$, which is associated with the vector $\left[\begin{matrix}1\\1\end{matrix}\right]$, and $y'=2S_2$, which is associated with the vector $\left[\begin{matrix}0\\2\end{matrix}\right]$.  Viewing the complexes as vectors, the \textit{reaction vector} associated to the reaction $y\to y'$ is simply $y'-y\in \Z^n$, which gives the state update of the system due to one occurrence of the reaction.

\begin{definition}
For $n \ge 0$, let $\mathcal{S} =\{S_1,...,S_n\}$, $\mathcal{C}=\cup_{y\to y'}\{y,y'\}$, and $\mathcal{R}=\cup_{y\to y'}\{y\to y'\}$ be the sets of species,  complexes, and reactions respectively. The triple $\{\mathcal{S,C,R}\}$ is called a \textit{reaction network}.  When $n = 0$, in which case $\S = \C = \Re = \emptyset$, the network is termed the \textit{empty network}. \hfill $\triangle$
\end{definition}

\begin{remark}\label{remark:reactionsspecify}
It is common to assume, and we shall do so throughout, that each species of a given reaction network appears with a positive coefficient in at least one complex, and each complex takes part in at least one reaction (as either a source or a product complex).  Thus, a reaction network $\{\S,\C,\Re\}$ is fully specified if we know $\Re$. In this case, we call $\S$ and $\C$ the set of species and the set of complexes \emph{associated with} $\Re$.
\end{remark}

To each reaction network $\{\mathcal{S,C,R}\}$, there is a unique directed graph constructed in the obvious manner: the vertices of the graph are given by $\C$ and a directed edge is placed from $y$ to $y'$ if and only if $y \to y' \in \mathcal{R}$.  Figure \ref{figure1} is an example of such a graph.  
We denote by $\ell$ the number of  connected components of the graph.  Note that by definition the directed graph associated to a reaction network contains only vertices corresponding to elements in $\C$ involved in some reaction, i.e., the degree of all vertices is at least 1 and so isolated vertices are not present in the associated network. 

\begin{remark}\label{remark:LCbound}
Note that since each connected component must consist of at least two vertices, we have the bound $\ell \le \frac{|\mathcal{C}|}{2}$. 
\end{remark}

It is a common practice, which we use here, to specify a reaction network by writing all the reactions, since the sets $\S$, $\C$, and $\Re$ are contained in this description.  For example, the reaction network in Figure \ref{figure1} has the set of species $\S=\{S_1,S_2\}$, the set of vertices $\C=\{\emptyset,S_2,S_1+S_2,2S_1,2S_2\}$ and the set of reactions $\Re=\{\emptyset\to S_1+S_2, S_1+S_2\to S_2, S_2\to \emptyset, 2S1\to 2S_2, 2S_2\to 2S_1\}$.

\begin{definition}
 The linear subspace $S=\text{span}\{y'-y\}$ generated by all reaction vectors is called the \textit{stoichiometric subspace} of the network. Denote $\dim(S)=s$.\hfill $\triangle$
\end{definition}

\begin{definition} 
A vertex, $y \in \Z^n_{\ge 0}$, is called \textit{binary} if $\sum_{i = 1}^n y_i = 2$. A vertex is called \textit{unary} if $\sum_{i = 1}^n y_i = 1$. The vertex  $\vec{0} \in \Z^n$ is said to be of \textit{zeroth order}. \hfill $\triangle$
\end{definition}

\begin{definition}
A reaction network $\{\mathcal{S,C,R}\}$ is called \textit{binary} if each vertex is binary, unary, or of zeroth order.\hfill $\triangle$
\end{definition}

In later sections, we will focus on \textit{binary} reaction networks due to their relevancy in the chemistry and biology literature.

\subsection{Deficiency and related results}
\begin{definition}
The \textit{deficiency} of a reaction network $\{\mathcal{S,C,R}\}$ is $\delta=|\mathcal{C}|-\ell-s$, where $|\mathcal{C}|$ is the number of vertices, $\ell$ is the number of connected components of the associated graph, and $s = \dim(\text{span}\{y'-y: y\to y'\in \Re\})$ is the dimension of the stoichiometric subspace of the network.

For each $j\le \ell$, we let $\C_j$ denote the collection of vertices in the $j$th connected component, $s_j$ be the corresponding dimension of the span of the reaction vectors of that component, and  define $\delta_j = |\C_j| - 1 - s_j$ to be the deficiency of that component.\hfill $\triangle$
\end{definition}

\begin{remark}\label{remark:empty}
From the definition of deficiency, the empty network has deficiency zero. 
\end{remark}

As we are interested in networks with deficiency zero, it is an important fact that a deficiency zero network cannot have too many vertices. The following lemma from \cite{prevalence} gives an upper bound.

\begin{lemma}[Lemma 2.1(e) from \cite{prevalence}]\label{lemma:Cbound}
Let $n\in \N$ and let $\{\S,\C,\Re \}$ be a reaction network with $n$ species. Assume that the reaction network has deficiency zero, then 
\[
|\C| \leq 2n.
\]
\end{lemma}

Next, we will provide another useful fact about deficiency from \cite{prevalence}: the deficiency of a reaction network cannot decrease if we add a reaction to it. This means deficiency zero is a monotone decreasing property, which guarantees that a threshold function for deficiency zero exists (see \cite{threshold}).

\begin{lemma}[Lemma 2.1(g) from \cite{prevalence}]\label{lemma:add_reaction}
Let $R=\{\S,\C,\Re\}$ and $\widehat{R}=\{\widehat{\S},\widehat{\C},\widehat{\Re}\}$ be two reaction networks with $\widehat{\Re}\setminus \Re = \{y\to y'\}$, a single reaction. Then
\[
\delta_{\widehat{R}} \geq \delta_{R}.
\]
\end{lemma}

\begin{definition}\label{def:89768968}
Let $R=\{\S,\C,\Re\}$ be a reaction network, and $\tilde{\Re}\subset \Re$. Then we denote by $\pi_{\tilde{\Re}}(R)$ the reaction network whose set of reactions is $\tilde{\Re}$, and whose species and vertices are the subsets of $\S$ and $\C$ that are associated with $\tilde{\Re}$, according to Remark \ref{remark:reactionsspecify}.  
\end{definition}

Note that in Definition \ref{def:89768968}, $\pi_{\tilde{\Re}}(R)$ can be thought of as a ``sub-network", or a projection of $R$ onto the subset of species, vertices, and reactions associated with $\tilde{\Re}$. The following corollary is a direct consequence of Lemma \ref{lemma:add_reaction}.

\begin{corollary}\label{cor:subnetwork}
Let $R=\{\S,\C,\Re\}$ be a reaction network, and $\tilde{\Re}\subset \Re$. Then
\[
\delta_{\pi_{\tilde{\Re}}(R)} \leq \delta_{R}.
\]
In particular, if $\pi_{\tilde{\Re}}(R)$ has a positive deficiency, then $R$ also has a positive deficiency.
\end{corollary}

We introduce two more Lemmas related to the deficiency of a network. Their proofs are similar to the proof of Lemma \ref{lemma:add_reaction}, and thus they are omitted for the sake of brevity.  The first lemma is well-known.

\begin{lemma}\label{lemma:unary}
Let $R=\{\S,\C,\Re\}$ be a reaction network whose vertices are either unary or of zeroth order, then $\delta_R=0.$
\end{lemma}
\begin{lemma}\label{lemma:independent}
Let $R=\{\S,\C,\Re\}$ be a reaction network and let $\tilde \Re$ be a subset of $\Re$ in which precisely one reaction of each reversible pair is removed.  If $\tilde \Re$ consists of linearly independent reaction vectors,  then $\delta_R=0$.
\end{lemma}

We illustrate the concept of deficiency with some reaction networks taken from the biology and chemistry literature.
\begin{example}[Enzyme kinetics \cite{ACK:product}]	
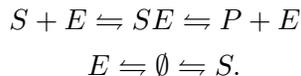
\begin{align*}
S+E &\leftrightharpoons  SE \leftrightharpoons P+E\\
&E \leftrightharpoons \emptyset \leftrightharpoons S.
\end{align*}
In this example, the reaction network has $|\mathcal{C}|=6$ vertices, there are $l=2$ connected components, and the dimension of the stochiometric subspace is $s=4$. Thus the deficiency is
\[
\delta = 6-2-4=0.
\]
\hfill $\triangle$
\end{example}
\begin{example}[Futile cycle enzyme \cite{WS2008}]
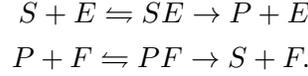
\begin{align*}
S+E \leftrightharpoons SE \rightarrow P+E\\
P+F \leftrightharpoons PF \rightarrow S+F.
\end{align*}
In this example, the reaction network has $|\mathcal{C}|=6$ vertices, there are $l=2$ connected components and the dimension of the stochiometric subspace can be calculated, which yields $s=3$. Thus the deficiency is
\[
\delta = 6-2-3 =1.
\]  \hfill $\triangle$
\end{example}
\section{A stochastic block model framework for binary CRNs}\label{sec4}
In this section we setup a stochastic block model for generating random reaction networks.

Let the set of species be $\mathcal{S}=\{S_1,S_2,\dots,S_n\}$. We consider binary reaction networks with species in $\mathcal{S}$. The set of all possible vertices is then
\[
\C^0_n = \{\emptyset, S_i, S_i+S_j:  \text{for  $1 \le i \le n$ and $1 \le j \le n$}.\}
\] 
For a given $n$, we denote  $N_n = |\mathcal{C}_n^0|$, the cardinality of $\mathcal{C}_n^0$.  Thus, $N_n$ is  the total number of possible unary, binary, and zeroth order vertices that can be generated from $n$ distinct species.  A straightforward calculation yields
\[
	N_n=1+n+n+\frac{n(n-1)}{2}= \frac{n^2+3n+2}{2},
\]
and so
\[
n \sim \sqrt{2N_n}.
\]

\begin{definition}
We denote by $E^{0,1}_n, E^{0,2}_n, E^{1,1}_n, E^{1,2}_n, E^{2,2}_n$ the sets of edges, or reactions, as follows: 
\begin{align*}
&E^{0,1}_n=\{\emptyset \leftrightharpoons S_i: 1\leq i\leq n\}\\
&E^{0,2}_n=\{\emptyset \leftrightharpoons S_i+S_j: 1\leq i,j\leq n\}\\
&E^{1,1}_n=\{S_i\leftrightharpoons S_j: 1\leq i,j\leq n; i\neq j\}\\
&E^{1,2}_n=\{S_i\leftrightharpoons S_j+S_k: 1\leq i,j,k\leq n\}\\
&E^{2,2}_n=\{S_i+S_j\leftrightharpoons S_h+S_k: 1\leq i,j,k,h\leq n; (i,j)\neq (k,h); (i,j) \neq (h,k)\}.
\end{align*}
\end{definition}
\begin{remark}\label{sizeEij}
$E^{0,1}_n, E^{0,2}_n, E^{1,1}_n, E^{1,2}_n, E^{2,2}_n$ completely partition the set of all possible edges. Note that $|E^{0,1}_n|\sim n$, $|E^{1,1}_n| \sim |E^{0,2}_n|\sim n^2$, $|E^{1,2}_n|\sim n^3$ and $|E^{2,2}_n|\sim n^4$. In fact, we have $|E^{i,j}_n|\sim n^{i+j}$.  Finally, note that the terms \emph{edges} and \emph{reactions} can be used interchangeably in the present context.
\end{remark}

We then consider a randomly generated network $G(N_n,p_n)$, which we will simply denote $G_n$ throughout, where the set of vertices is the set of vertices $\mathcal{C}_n^0$, and the probability that there is an edge between two vertices is given as follows
\begin{enumerate}
    \item an edge in $E^{0,1}_n$ appears in the random graph with probability $p_n^{0,1}=n^{\alpha_{0,1}}p_n$,
    \item an edge in $E^{0,2}_n$ appears in the random graph with probability $p_n^{0,2}=n^{\alpha_{0,2}}p_n$,
    \item an edge in $E^{1,1}_n$ appears in the random graph with probability $p_n^{1,1}=n^{\alpha_{1,1}}p_n$,
    \item an edge in $E^{1,2}_n$ appears in the random graph with probability $p_n^{1,2}=n^{\alpha_{1,2}}p_n$,
    \item an edge in $E^{2,2}_n$ appears in the random graph with probability $p_n$,
\end{enumerate}
where $\alpha_{0,1},\alpha_{0,2},\alpha_{1,1},\alpha_{1,2}$ are parameters that can be used to control the structure of the random graph. 
Each random graph now corresponds to a reaction network in the following way,
\begin{enumerate}
\item each vertex with positive degree in the random graph represents a vertex in the reaction network graph, and
\item each edge in the random graph represents a reaction in the reaction network graph. We can assume all reactions are reversible, i.e., that $y\to y'\in \Re \implies y'\to y \in \Re$, since deficiency does not depend on the direction of the edges.
\end{enumerate}

We will denote the reaction network associated with the graph $G(N_n,p_n)$ by $R(N_n,p_n)$, which we will often simplify to $R_n.$  We will denote the deficiency of $R_n$ by $\delta_{R_n}.$
\begin{remark}
In later sections it will be more useful to work with the expected and actual number of edges in each set $E^{i,j}_n$ instead of $p^{i,j}_n$. Thus, for convenience we denote by $M_{i,j}(n)$ the number of realized edges from $E^{i,j}_n$ and by $K_{i,j}(n)=\E [M_{i,j}(n)]$  the expected number of realized edges from $E^{i,j}_n$. It is straightforward to see that $M_{i,j}(n)$ has a binomial distribution, and from Remark \ref{sizeEij} that 
\[
K_{i,j}(n)\sim n^{i+j}n^{\alpha_{i,j}}p_n
\]
for $(i,j)\neq (2,2)$ and $K_{2,2}(n)=n^4p_n$. 
\end{remark}
With the stochastic block model above, we can model a wide range of reaction networks by tweaking the parameters $\{\alpha_{i,j}\}$. Next, we provide a few examples to illustrate this flexibility. 
\begin{example}[The case $\alpha_{0,1}=\alpha_{0,2}=\alpha_{1,1}=\alpha_{1,2}=0$]

In this case, we recover the unweighted Erd\H os-R\'enyi framework in \cite{prevalence}. From \cite{prevalence}, and Theorem \ref{theorem:pn} below, the threshold function for deficiency zero is $r(n) = \frac1{n^3}$. In other words,
\[
\lim_{n\to\infty}P(\delta_{R_n}=0) = \begin{cases}
0 \quad \text{when} \quad \lim_{n\to\infty}\frac{p_n}{r(n)} = \infty\\
1 \quad \text{when} \quad \lim_{n\to\infty}\frac{p_n}{r(n)} = 0
\end{cases}.
\]
Lemma 4.2 and Lemma 4.3 in \cite{prevalence} then tell us that for $\lim_{n\to\infty}\frac{p_n}{r(n)} = 0$, the random reaction networks we observe only contain edges from $E^{2,2}_n$ with high probability. In other words, with the unweighted framework, we only see deficiency zero in ``closed systems" (reaction networks with no inflow and outflow) of a very particular type. Reactions such as inflow and outflow, unary-unary, and unary-binary  are underrepresented in this case.
\hfill $\triangle$
\end{example}

\begin{example}[A closed system with $\alpha_{0,1}=\alpha_{0,2}=0, \alpha_{1,1}=2, \alpha_{1,2}=1$]

In this case, we have the expected number of edges in $E^{0,1}_n$ is $K_{0,1}(n)\sim np_n$ and the expected number of edges in $E^{0,2}_n$ is $K_{0,2}(n) \sim n^2p_n$. It is easy to check that the parameters $\alpha_{i,j}$ are selected such that 
\[
K_{1,1}(n)\sim K_{1,2}(n)\sim K_{2,2}(n) \sim n^4p_n \quad \text{and} \quad K_{0,1}(n), K_{0,2}(n) \ll n^4p_n.
\]
Thus the random reaction networks we observe will have similar expected amount of reactions in $E^{1,1}_n$, $E^{1,2}_n$, $E^{2,2}_n$.  We also have  that the expected number of reactions in $E^{0,1}_n$ and $E^{0,2}_n$ is significantly less. In particular, if $p_n\ll \frac{1}{n^2}$, the probability of seeing any reaction in $E^{0,1}_n$ and $E^{0,2}_n$ goes to $0$ as $n\to\infty$.  Hence, the random networks we observe will not have inflow and outflow reactions with high probability. Thus, this scheme is suitable to model closed systems without underrepresenting unary-unary and unary-binary reactions, unlike the case in Example 3. From Theorem \ref{theorem:pn} below, the threshold function for this case is $r(n) = \frac{1}{n^3}$ \hfill $\triangle$
\end{example}

\begin{example}[An open system with $\alpha_{0,1}=3$, $\alpha_{1,1}=\alpha_{0,2}=2$, $\alpha_{1,2}=1$]

In this case, the expected number of realized edges $K_{i,j}(n) \sim n^4p_n$ for all $(i,j)$. Thus, this scheme is suitable to model an ``open system" with inflow and outflow reactions, and with similar amount of reactions from each type. See Figure \ref{fig:realization} for a realization of this system with a specific choice of parameters. From Theorem \ref{theorem:pn} below, the threshold function for this case is $r(n) = \frac{1}{n^{10/3}}$.
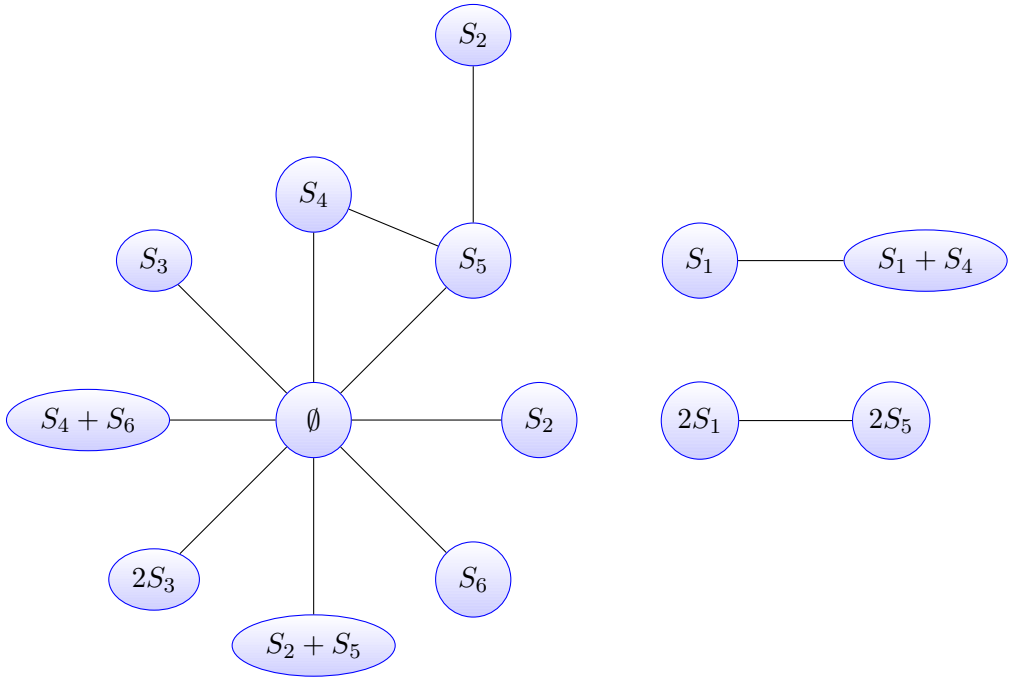
\begin{figure}[h]
\begin {tikzpicture}[node distance=3cm, state1/.style ={ circle ,top color =white , bottom color = blue!20 ,draw,blue , text=black , minimum width =1 cm},state2/.style ={ ellipse ,top color =white , bottom color = blue!20 ,draw,blue , text=black , minimum width =1 cm}]
\node[state1] (0) {$\emptyset$};
\node[state1] (S4) [above of= 0] {$S_4$};
\node[state1] (S5) [above right of=0] {$S_5$};
\node[state1] (S2) [right of=0] {$S_2$};
\node[state1] (S6) [below right of=0] {$S_6$};
\node[state2] (S2+S5) [below of= 0] {$S_2+S_5$};
\node[state2] (2S3) [below left of=0] {$2S_3$};
\node[state2] (S4+S6) [left of=0] {$S_4+S_6$};
\node[state2] (S3) [above left of=0] {$S_3$};
\draw (0) -- (S4);
\draw (0) -- (S5);
\draw (0) -- (S2);
\draw (0) -- (S6);
\draw (0) -- (S2+S5);
\draw (0) -- (2S3);
\draw (0) -- (S4+S6);
\draw (0) -- (S3);
\node[state2] (2S2) [above of=S5] {$S_2$};
\draw (2S2) -- (S5);
\draw (S4) -- (S5);

\node[state1] (S1) [right= 2cm of S5] {$S_1$};
\node[state2] (S1+S4) [right of=S1] {$S_1+S_4$};
\draw (S1) -- (S1+S4);

\node[state1] (2S1) [below=1.1cm of S1] {$2S_1$};
\node[state1] (2S5) [right=1.5cm of 2S1] {$2S_5$};
\draw (2S1) -- (2S5);

\end{tikzpicture}
\caption{A realization of the open system in Example 5 with $n=6$ and $p=\frac{0.8}{n^3}$. \textit{Note: The figure only includes non-isolated vertices.}}
\label{fig:realization}
\end{figure}
\hfill $\triangle$
\end{example}

\section{Conditions for deficiency zero in terms of $K_{i,j}(n)$}\label{sec5}

In this section, we will provide a set of conditions on $K_{i,j}(n)$ that guarantee $\lim_{n\to \infty}\P(\delta_{R_n}=0)=0$. We will also show that under the ``converse" of these conditions, $\lim_{n\to \infty}\P(\delta_{R_n}=0)=1$. These conditions are used in Section \ref{sec6} to form an algorithm to find the threshold function for deficiency zero. Specifically, given any choice of $\{\alpha_{i,j}\}$, the algorithm provides a single threshold function $r(n)$ for deficiency zero.

\subsection{Conditions on $K_{i,j}(n)$ for  $\lim_{n\to \infty}\P(\delta_{R_n}=0)=0$}
We start this section by providing some examples which illustrate different ways to break deficiency zero.
\begin{example}
Consider a reaction network with only 2 species $\S=\{S_1,S_2\}$
\begin{align*}
&\S_1 \leftrightharpoons S_2\\
&\S_1+S_2 \leftrightharpoons \emptyset.
\end{align*}
The reaction network has deficiency 
\[
\delta = |\C|-\ell-s = 4- 2 -2 =0.
\]
However, since there are only 2 species, we must have $s\leq 2$. Thus if we add more vertices and reactions, it is easy to get a positive deficiency from the new reaction network. For example, if we add $2S_1\leftrightharpoons S_2$, then the new network is 
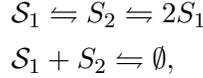
\begin{align*}
&\S_1 \leftrightharpoons S_2 \leftrightharpoons 2S_1\\
&\S_1+S_2 \leftrightharpoons \emptyset,
\end{align*}
and the new deficiency is $\delta' = |\C'|-\ell'-s' = 5-2-2= 1$. In this example, we break deficiency zero by having too many vertices with respect to the number of species. \hfill $\triangle$
\end{example}
\begin{example}
Consider a reaction network with $10$ species $\S=\{S_1,\dots,S_{10}\}$, which is given below
\begin{align*}
&\S_1 \leftrightharpoons S_2 \leftrightharpoons \dots \leftrightharpoons S_{10}.
\end{align*}
The reaction network has deficiency
\[
\delta = |\C|-\ell-s = 10- 1 - 9 =0.
\]
Note that all unary vertices are already in the network, and the dimension of the stochiometric subspace, which is 9, is nearly at the maximum possible value of 10. If we add one or two more reactions in $E^{1,2}_n,E^{0,2}_n,$ or $E^{2,2}_n$, then it is easy to break deficiency zero since the dimension of the original network is almost at its maximum. For example, if we add $S_1+S_2\to S_3+S_4$, then the new network is 
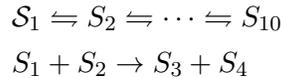
\begin{align*}
&\S_1 \leftrightharpoons S_2 \leftrightharpoons \dots \leftrightharpoons S_{10}\\
& S_1+S_2\to S_3+S_4
\end{align*}
and the new deficiency is $\delta' = |\C'|-\ell'-s' = 12 - 2 - 9 = 1$. In this example, we break deficiency zero by adding too many more reactions when the dimension of the stoichiometric subspace is already nearly full from the unary reactions. \hfill $\triangle$
\end{example}
\begin{example}
Consider a reaction network with $10$ species $\S=\{S_1,\dots,S_{10}\}$, and a high number of reactions in $E^{0,1}_n$
\begin{align*}
\emptyset \leftrightharpoons S_i \quad \text{where} \quad i=1,\dots,8.
\end{align*}
The reaction network has deficiency 
\[
\delta = |\C|-\ell-s = 9- 1 - 8 =0.
\]
If we add a high enough number of reactions in $E^{1,2}_n,E^{0,2}_n,$ or $E^{2,2}_n$, then it is likely that we add a reaction whose species are in $\{S_1,\dots,S_8\}$, which breaks deficiency zero. For example, consider the new network
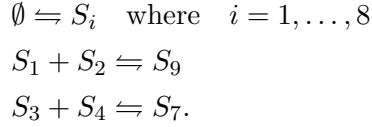
\begin{align*}
&\emptyset \leftrightharpoons S_i \quad \text{where} \quad i=1,\dots,8\\
& S_1+S_2\leftrightharpoons S_9\\
& S_3+S_4\leftrightharpoons S_7.
\end{align*}
The new deficiency is $\delta' = |\C'|-\ell'-s' = 12 - 2 - 9 = 1$. In this example, we break deficiency zero by having a high number of reaction in $E^{0,1}_n$ and a high enough number of reaction in $E^{1,2}_n,E^{0,2}_n,$ or $E^{2,2}_n$.  \hfill $\triangle$
\end{example}
It turns out that the three examples above are representative of all cases when we have $\lim_{n\to \infty}\P(\delta_{R_n}=0)=0$. We provide rigorous conditions in the following theorem. 
\begin{theorem}\label{main1}
If one of the following conditions holds, then $\lim_{n\to \infty}\P(\delta_{R_n}=0)=0$.
\begin{enumerate}[leftmargin]
    \item[(C1.1)] Either $K_{0,2}(n) \gg n$, $K_{1,2}(n) \gg n$, or $K_{2,2}(n) \gg n$. 
    \item[(C1.2)] $K_{1,1}(n)\gg n$ and either $K_{0,2}(n) \gg 1$, $K_{1,2}(n) \gg 1$, or $K_{2,2} (n)\gg 1$.
    \item[(C1.3)] Either $K_{0,1}(n)^2K_{0,2}(n) \gg n^2$, $K_{0,1}(n)^3K_{1,2}(n) \gg n^3$, or $K_{0,1}(n)^4K_{2,2}(n) \gg n^4$.
\end{enumerate}
\end{theorem}
\begin{remark}
In Theorem \ref{main1}, the three conditions are not purely technical; there is intuition behind each condition as described in the examples at the beginning of this section, and below.
\begin{enumerate}
    \item Condition C1.1 refers to the case when there are too many vertices in the reaction network, which makes its deficiency strictly positive (see Lemma \ref{lemma:Cbound}). Note that $K_{0,1}(n) \gg n$ and $K_{1,1}(n) \gg n$ can not break deficiency zero in this regard. Obviously, it is impossible to have $K_{0,1}(n) \gg n$ since $|E^{0,1}_n|=n$. The condition $K_{1,1}(n) \gg n$ by itself still results in the network being deficiency zero (see Lemma \ref{lemma:unary}). However, the condition $K_{1,1}(n) \gg n$ together with a non-trivial number of reactions from $E^{0,2}_n, E^{1,2}_n, E^{2,2}_n$ can break deficiency zero. This is stated formally in Condition C1.2.
    
    \item Condition C1.2 refers to the case when the dimension of the stochiometric subspace $s$ is almost fully exhausted from reactions in $E^{1,1}_n$. Recall that $\delta = |\C|-\ell-s$, so in this case as we add more reactions in $E^{0,2}_n, E^{1,2}_n, E^{2,2}_n$ , $|\C|-\ell$ increases but $s$ does not, making the deficiency positive.
    \item Condition C1.3 refers to the case where there is a high probability of some inflow or outflow reaction in $E^{0,1}_n$ and a reaction in another edge set being linearly dependent, which in turn makes the deficiency positive. It will also be apparent later that having a nontrivial number of inflow or outflow reactions in $E^{0,1}_n$ makes it more difficult to have deficiency zero.
\end{enumerate}
\end{remark}

We prove the theorem  via a series of lemmas. We begin by showing that if Condition (C1.1) holds, then $\lim_{n\to \infty}\P(\delta_{R_n}=0)=0$. 

\begin{lemma}
If either $K_{0,2}(n) \gg n$, $K_{1,2}(n) \gg n$, or $K_{2,2}(n) \gg n$, then we have
\[
\lim_{n\to \infty}\P(\delta_{R_n}=0)=0.
\]
\end{lemma}
\begin{proof} 
Recall from Lemma \ref{lemma:Cbound} that there cannot be too many 
vertices in a network with deficiency zero. In particular, $\delta_{R_n}=0$ implies $|\C| \leq 2n$. We will argue that in each of the three cases the number of non-isolated vertices in $G_n$, which correspond with the vertices of the associated reaction network $R_n$, is likely to be much higher than the bound $2n$, implying the network has positive deficiency.
The first case is straightforward, and the remaining two cases follow the same technique as Lemma 4.1 and Theorem 4.1 in \cite{prevalence}.
\begin{enumerate}
\item First, we assume that $K_{0,2}(n) \gg n$.
From Lemma \ref{lemma:Cbound}, we have that $\delta_{R_n}=0$ implies $|\C| \leq 2n$, which in turns implies $M_{0,2}(n) \leq 2n-1$.  Thus
\begin{align*}
\P(\delta_{R_n} =0) &\leq \P(M_{0,2}(n) \leq 2n-1)\\
& = \P(K_{0,2}(n) - M_{0,2}(n) \geq K_{0,2}(n) - (2n-1))\\
&\leq \frac{\var (M_{0,2}(n))}{(K_{0,2}(n) - (2n-1))^2}.
\end{align*}
Since $M_{0,2}(n)$ has a binomial distribution, $\var (M_{0,2}(n)) \leq \E[M_{0,2}(n)]=  K_{0,2}(n)$. Together with the fact that $K_{0,2}(n) \gg n$, we have $\frac{\var (M_{0,2}(n))}{(K_{0,2}(n) - (2n-1))^2}\to 0$, as $n\to\infty$, and thus $\lim_{n\to \infty}\P(\delta_{R_n}=0)=0$.

\item 
Next, we assume  $K_{1,2}(n) \gg n$.  We observe that based on Corollary \ref{cor:subnetwork}, $\delta_{R_n}=0$ must imply $\delta_{\pi_{E^{1,2}_n}(R_n)} =0$, where, recalling Definition \ref{def:89768968}, $\pi_{E^{1,2}_n}(R_n)$ is the subnetwork of $R_n$ with reactions in $E^{1,2}_n$. Thus we have
\[
\P(\delta_{R_n}=0) \leq  \P(\delta_{\pi_{E^{1,2}_n}(R_n)} =0).
\]
Again, we make use of the upper bound in Lemma \ref{lemma:Cbound}. $\delta_{\pi_{E^{1,2}_n}(R_n)} =0$ must imply the number of non-isolated vertices in $\pi_{E^{1,2}_n}(R_n)$ is bounded by $2n$. Let $I$ be the set of isolated binary vertices in $\pi_{E^{1,2}_n}(R_n)$. Since there are $\frac{n(n+1)}{2}$ binary vertices, we must then have
\[
|I| > \frac{n(n+1)}{2} - 2n,
\]
and as a result
\[
\P(\delta_{R_n}=0)  \leq \P\bigg(|I| > \frac{n(n+1)}{2} - 2n\bigg).
\]
The probability that a binary vertex is isolated in $\pi_{E^{1,2}_n}(R_n)$ is $(1-p^{1,2}_n)^n$, because there are precisely $n$ unary vertices.  Thus, summing over the binary vertices yields
\[
\E|I| = \frac{n(n+1)}{2}(1-p^{1,2}_n)^n.
\]
We can also derive $\var(|I|)$   since $|I|$ is binomially distributed. Using $\E|I|$ and $\var(|I|)$, a rigorous proof for
\[\lim_{n\to\infty}\P\bigg(|I| > \frac{n(n+1)}{2} - 2n\bigg) =0\]
can be carried out by precisely the same argument as Lemma 4.1 in \cite{prevalence}. We omit it for the sake of brevity. 

\item Finally, we assume $K_{2,2}(n) \gg n$.  We observe that based on Corollary \ref{cor:subnetwork}, $\delta_{R_n}=0$ must imply $\delta_{\pi_{E^{2,2}_n}(R_n)} =0$, where $\pi_{E^{2,2}_n}(R_n)$ is the subnetwork of $R_n$ with reactions in $E^{2,2}_n$. Thus we have
\[
\P(\delta_{R_n}=0) \leq  \P(\delta_{\pi_{E^{2,2}_n}(R_n)} =0).
\]
Note that $K_{2,2}(n) \gg n$ implies $n^4p_n \gg n$, and thus $p_n\gg \frac{1}{n^3}$. The remainder of the proof follows along the same lines as the proof of Lemma 4.1 and Theorem 4.1 in \cite{prevalence}. \qedhere 
\end{enumerate}
\end{proof} 

The following proposition will be useful in the proof that Condition C1.2 implies $\lim_{n\to \infty}\P(\delta_{R_n}=0)=0$.

\begin{proposition}\label{prop:unary+2}
Let $R=\{\S,\C,\Re\}$ be a reaction network with  $\S=\{S_1,S_2,\dots,S_n\}$. Assume that all vertices in $R$ are unary, and $R$ has only one connected component. Let $i,j,p,q\in \{1,\dots,n\}$ be such that $\{i,j\} \ne \{p,q\}$, and let $\widehat{\Re}=R\cup\{\emptyset\leftrightharpoons S_i+S_j,\emptyset\leftrightharpoons S_p+S_q\}$  and  $\widehat{R}$ be the reaction network associated with $\widehat\Re$. Then  $\delta_{\widehat{R}}=1$.
\end{proposition}

Note that in the above proposition we are allowing $i = j$ and/or $i = p$. 
\begin{proof}
Due to Lemma \ref{lemma:unary}, the deficiency of $R$ is necessarily zero (since it contains only unary vertices).  
Starting from $R$, adding the pair of reversible reactions $\emptyset\leftrightharpoons S_i+S_j,\emptyset\leftrightharpoons S_p+S_q$ to form $\widehat{R}$ increases the number of vertices by three, and increases the number of connected components by 1. It is straightforward to check that since  the vertices $\{S_i,S_j,S_p,S_q\}$  are contained within $\C$, the addition of the reaction vectors for $\emptyset \leftrightharpoons S_i+S_j$ and $\emptyset \leftrightharpoons S_p+S_q$ only increases the size of the dimension of the stoichiometric subspace by 1.  Hence, we  have $\delta_{\widehat{R}} = \delta_{R} + 3 - 1 - 1 = 1$. 
\end{proof}

We now show that Condition (C1.2) yields the desired result.
\begin{lemma}
If $K_{1,1}(n)\gg n$ and either $K_{0,2}(n) \gg 1$, $K_{1,2}(n) \gg 1$, or $K_{2,2}(n)\gg 1$, then we have 
\[
\lim_{n\to \infty}\P(\delta_{R_n}=0)=0.
\]
\end{lemma}
\begin{proof}
Suppose $K_{0,2}(n) \gg 1$. The other two cases can be handled in a same manner. 

$M_{0,2}(n)$ is binomially distributed with mean $K_{0,2}(n) \gg 1$. Thus, standard methods show
\[
\lim_{n\to\infty}\P(M_{0,2}(n) \geq 2)=1.
\]
Now it suffices to show 
\[
\lim_{n\to \infty}\P(\delta_{R_n}=0, M_{0,2}(n) \geq 2)=0.
\]
Let $G_n^{1,1}$ be the subgraph of $G_n$ consisting of all vertices $S_i$ (even those that are isolated)  and all edges in $E^{1,1}_n$ that are realized in $G_n$. Let $B_n$ be the largest component in $G_n^{1,1}$ and let $|B_n|$ be its size (number of vertices). When $M_{0,2}(n) \ge 2$, we let $B_n^+$ be the union of $B_n$ with two edges chosen uniformly at random from $E^{0,2}_n$ that are realized in $G_n$.  If  $M_{0,2}(n) \le 1$ we choose the two reactions uniformly at random from $E^{0,2}_n$.  We denote the chosen two edges by $\emptyset \leftrightharpoons S_i+S_j, \emptyset\leftrightharpoons S_p+S_q$ and note that $\{i,j\} \ne \{p,q\}$. Note that by symmetry the distribution  of the pair $(\emptyset \leftrightharpoons S_i+S_j, \emptyset\leftrightharpoons S_p+S_q)$ is the same as if we simply chose two reactions from $E^{0,2}_n$ uniformly at random.   Since $\delta_{B_n^+} \le \delta_{R_n}$, we must have
\[
\P(\delta_{R_n}=0,M_{0,2}(n)\geq 2) \le  \P(\delta_{B_n^+}=0,M_{0,2}(n)\geq 2) \le \P(\delta_{B_n^+}=0).
\]
Thus it suffices to show 
\[
\lim_{n\to \infty}\P(\delta_{B_n^+}=0)=0.
\]
Conditioning on the size of $B_n$, the largest component of $G^{1,1}_n$, yields
\begin{equation}\label{eq2398402340}
\P(\delta_{B_n^+}=0) = \sum_{k = 1}^n \P(\delta_{B_n^+}=0| |B_n|=k)\P(|B_n| = k).
\end{equation}

\noindent From Proposition \ref{prop:unary+2}, we know that if $B_n^+$ has a deficiency of zero, then not all of $S_i,S_j,S_p,S_q$ are contained in $B_n$.  Thus we have
\begin{align}\label{eq23jio234jko23}
\begin{split}
    \P(\delta_{B_n^+}=0| |B_n|=k) &\le \P(\text{not all of $S_i,S_j,S_p,S_q$ are contained in $B_n$} ||B_n|=k)\\
    &=1 - \P(S_i,S_j,S_p,S_q\in B_n \large| |B_n|=k).
    \end{split}
\end{align}
We will compute the probability as follows
\begin{align}\label{eq:87870}
    \P(S_i,S_j,S_p,S_q\in B_n \large| |B_n|=k) &= \P(S_p,S_q\in B_n \large| S_i,S_j\in B_n, |B_n|=k)\P(S_i,S_j\in B_n | |B_n| = k).
\end{align}
We first consider the probability $\P(S_i,S_j\in B_n | |B_n| = k)$.  Since $|B_n| = k$, there are exactly $\binom{k}{2}$ ways of choosing a reaction of the form $\emptyset \leftrightharpoons S_i+S_j$ with $i \ne j$ and $S_i,S_j \in B_n$. Similarly, for the case $i =j$, there are exactly $k$ ways of choosing a reaction of the form $\emptyset \leftrightharpoons 2S_i$ with $S_i \in B_n$.  Since there are a total of $\binom{n}{2} + n$ elements in $E_{n}^{0,2}$ we have 
\begin{equation}\label{eq2k34k2jl234kjl}
    \P(S_i,S_j\in B_n | |B_n| = k) = \frac{{k\choose 2}+k}{{n\choose 2}+n}=\frac{k(k+1)}{n(n+1)} \geq \bigg(\frac{k}{n}\bigg)^2,
\end{equation}
where the inequality holds since $k \le n$. Similarly, we have
\begin{equation}\label{eq3458912k3l4j}
\P(S_p,S_q\in B_n \large| S_i,S_j\in B_n, |B_n|=k) \frac{{k\choose 2}+k-1}{{n\choose 2}+n-1} \geq \bigg(\frac{k}{n}\bigg)^2,
\end{equation}
where the inequality holds for $k\geq 4$, which can be verified in a straightforward manner. From \eqref{eq23jio234jko23}, \eqref{eq:87870}, \eqref{eq2k34k2jl234kjl}, and \eqref{eq3458912k3l4j} we have that for $k\geq 4$
\begin{align}\label{eqjk23lk49898234}
\P&(\delta_{B_n^+}=0||B_n|=k) \leq 1-\bigg(\frac{k}{n}\bigg)^4.
\end{align}

\noindent Finally, combining \eqref{eq2398402340} and \eqref{eqjk23lk49898234}, we have
\begin{align}\label{eq234jkio234}
\lim_{n\to\infty}\P(\delta_{B_n^+}=0)&\leq \sum_{k=4}^n  \bigg(1-\bigg(\frac{k}{n}\bigg)^4\bigg)\P(|B_n|=k) + \sum_{k=1}^3\P(|B_n|=k)\nonumber \\
&=\E\bigg(1-\bigg(\frac{|B_n|}{n}\bigg)^4\bigg) + \sum_{k=1}^3\bigg(\frac{k}{n}\bigg)^4\P(|B_n|=k)\nonumber \\ 
&\leq 1- \bigg(\E\bigg(\frac{|B_n|}{n}\bigg)\bigg)^4 + \frac{98}{n^4},
\end{align}
where the last inequality is due to Jensen's inequality.
Since $K_{1,1}(n) \gg n$, we have the edge probability for the edges in $E^{1,1}_n$ satisfy
\[
p_n^{1,1} \sim \frac{K_{1,1}(n)}{n^2} \gg \frac{n}{n^2}=\frac{1}{n}.
\]
From, \cite{ER:giantcomponent}, $\frac{|B_n|}{n}-f(c_n)\overset{\P}{\to} 0$, where $f(c_n)=1-\frac{1}{c_n}\sum_{k=1}^\infty \frac{k^{k-1}}{k!}(c_ne^{-c_n})^k$ and $c_n=np^{1,1}_n$. Since $np_n^{1,1}\gg 1$, it is straightforward to verify that $\lim_{n\to\infty}f(c_n)=1$. Since both $\frac{|B_n|}{n}$ and $f(c_n)$ are bounded by $1$, we have
\[
\lim_{n\to\infty}\E\bigg(\frac{|B_n|}{n}\bigg) = \lim_{n\to\infty} f(c_n)=1,
\]
which completes the proof.
\end{proof}

Finally, we have the proof related to Condition C1.3.
\begin{lemma}
If either  $K_{0,1}(n)^2K_{0,2}(n) \gg n^2$, $K_{0,1}(n)^3K_{1,2}(n) \gg n^3$ or $K_{0,1}(n)^4K_{2,2}(n) \gg n^4$, then we have
\[
\lim_{n\to \infty}\P(\delta_{R_n}=0)=0.
\]
\end{lemma}
\begin{proof}
It suffices to show $K_{0,1}(n)^2K_{0,2}(n) \gg n^2$ implies $\lim_{n\to \infty}\P(\delta_{R_n}=0)=0.$ The other two cases follow the same argument. Recall that $M_{0,1}(n)$ has a binomial distribution with $|E^{0,1}_n|=n$ trials and mean $\E M_{0,1}(n)=K_{0,1}(n)$, and $M_{0,2}(n)$ is a binomial distribution with $|E^{0,2}_n|=n(n+1)/2$ trials and mean $\E M_{0,2}(n)=K_{0,2}(n)$. Thus we have 
\begin{equation}\label{eq12389751}
 \P(\delta_{R_n}=0) = \sum_{\substack{i\leq n\\ j\leq n(n+1)/2}}\P(\delta_{R_n}=0|M_{0,1}(n)=i,M_{0,2}(n)=j)\P(M_{0,1}(n)=i, M_{0,2}(n)=j).
\end{equation}
Consider the event $\delta_{R_n}=0$ conditioned on $M_{0,1}(n)=i, M_{0,2}(n)=j$. Note that a reaction network of the form $\emptyset \leftrightharpoons S_p, \quad \emptyset\leftrightharpoons S_q, \quad \emptyset \leftrightharpoons S_p+S_q$ has positive deficiency, so any network containing it also has positive deficiency according to Corollary \ref{cor:subnetwork}. Thus $\delta_{R_n}=0$ implies there is no such subnetwork in $R_n$. 

There are $\frac{n(n+1)}{2}$ reactions in $E^{0,2}_n$, thus the probability that there is reaction of the form  $\emptyset\leftrightharpoons S_p+S_q$ (note that $p$ and $q$ can be the same) where $\emptyset \leftrightharpoons S_p$ and $\emptyset \leftrightharpoons S_q$ are already present is
\[
\frac{{i\choose 2}+i}{\frac{n(n+1)}{2}}=\frac{i(i+1)}{n(n+1)}.
\]
We may then us a sequential argument (on the $j$ elements from $E^{0,2}_n$ that have been realized) that is similar to the one used around \eqref{eq3458912k3l4j} to conclude
\begin{equation}\label{eq12343213}
\P(\delta_{R_n}=0|M_{0,1}(n)=i,M_{0,2}(n)=j) \leq \bigg(1-\frac{i(i+1)}{n(n+1)}\bigg)^j.
\end{equation}
Combining \eqref{eq12389751} and \eqref{eq12343213}, we have
\begin{align*}
 \P(\delta_{R_n}=0) & \leq \sum_{\substack{i\leq n \\ j\leq n(n+1)/2}} \bigg(1-\frac{i(i+1)}{n(n+1)}\bigg)^j\P(M_{0,1}(n)=i, M_{0,2}(n)=j)\\
 &=\E \bigg[ \bigg(1-\frac{M_{0,1}(n)(M_{0,1}(n)+1)}{n(n+1)}\bigg)^{M_{0,2}(n)}\bigg].
\end{align*}
We have $\frac{M_{0,1}(n)(M_{0,1}(n)+1)}{n(n+1)} \geq \frac{M_{0,1}(n)^2}{2n^2}$, thus
\begin{align}\label{eq9017234}
 \P(\delta_{R_n}=0) & \leq  \E \bigg[ \bigg(1-\frac{M_{0,1}(n)^2}{2n^2}\bigg)^{M_{0,2}(n)}\bigg] \leq \E \bigg[ e^{-\frac{M_{0,1}(n)^2M_{0,2}(n)}{2n^2}}\bigg],
\end{align}
where the second inequality follows the fact that $1-x\leq e^{-x}$. Notice further that $e^{-x} \leq \frac{1}{x+1}$ for $x\geq 0$, hence we have
\begin{align}\label{eqoiuj90123}
 \E \bigg[ e^{-\frac{M_{0,1}(n)^2M_{0,2}(n)}{2n^2}}\bigg] \leq \E \bigg[\frac{2n^2}{{M_{0,1}(n)^2M_{0,2}(n)}+2n^2} \bigg]=2n^2\E\bigg[\frac{1}{M_{0,1}(n)^2M_{0,2}(n)+2n^2} \bigg].
\end{align}
Since $M_{0,1}(n)\leq n$ and $M_{0,2}(n)\leq \frac{n(n+1)}{2}$, we have for $n$ large enough 
\begin{align}
M_{0,1}(n)^2M_{0,2}(n)+2n^2 &\geq (M_{0,1}(n)^2+1)(M_{0,2}(n)+1) \nonumber \\
&\geq \frac{1}{2}(M_{0,1}(n)+1)^2(M_{0,2}(n)+1) \nonumber\\
&\geq \frac{1}{4}(M_{0,1}(n)+1)(M_{0,1}(n)+2)(M_{0,2}(n)+1), \label{eq1235134123}
\end{align}
where the first inequality can be verified by expanding the right hand side and utilizing the inequalities on $M_{0,1}(n)$ and $M_{0,2}(n)$, the second inequality follows by the well known $\frac12(a+b)^2 \le a^2  +b^2$  inequality, and the last inequality comes from $M_{0,1}(n)+1\geq \frac{1}{2}(M_{0,1}(n)+2)$, which is true as long as $M_{0,1}(n)\geq 0$.

Combining \eqref{eq9017234}, \eqref{eqoiuj90123}, \eqref{eq1235134123}, and noticing that $M_{0,1}(n)$ and $M_{0,2}(n)$ are independent, we have
\begin{align}\label{eqnjk12j3ou09}
\P(\delta_{R_n}=0) & \leq  8n^2 \E\bigg[\frac{1}{(M_{0,1}(n)+1)(M_{0,1}(n)+2)}\bigg]\E\bigg[\frac{1}{M_{0,2}(n)+1}\bigg].
\end{align}
Since $M_{0,1}(n)\sim B(n,K_{0,1}(n)/n)$, from Lemma \ref{lemma:binomial}, we have 
\begin{align}\label{eq19023jk}
    \E\bigg[&\frac{1}{(M_{0,1}(n)+1)(M_{0,1}(n)+2)}\bigg] \leq \frac{1}{K_{0,1}(n)^2}.
\end{align}
We also have $M_{0,2}(n)\sim B(n(n+1)/2,\frac{K_{0,2}(n)}{n(n+1)/2})$. Repeating the same argument as above, we have
\begin{align}\label{eqnj1238990123}
\E\bigg[\frac{1}{M_{0,2}(n)+1}\bigg] \leq \frac{1}{K_{0,2}(n)}.
\end{align}
Thus from \eqref{eqnjk12j3ou09}, \eqref{eq19023jk}, \eqref{eqnj1238990123} we have
\[
\P(\delta_{R_n}=0) \leq \frac{8n^2}{K_{0,1}(n)^2K_{0,2}(n)}.
\]
Since $K_{0,1}(n)^2K_{0,2}(n) \gg n^2$ the proof is complete.
\end{proof}
\subsection{Conditions on $K_{i,j}(n)$ for  $\lim_{n\to \infty}\P(\delta_{R_n}=0)=1$}

Note that the conditions below are essentially the converse of Theorem \ref{main1}.

\begin{theorem}\label{main2}
If all of the following conditions hold, then $\lim_{n\to \infty}\P(\delta_{R_n}=0)=1$.
\begin{enumerate}[leftmargin]
    \item[(C2.1)] $K_{0,2}(n) \ll n$, $K_{1,2}(n) \ll n$, and $K_{2,2}(n) \ll n$.
    \item[(C2.2)] One of the following conditions holds 
    \begin{itemize}
        \item[(C2.2.1)] $K_{1,1}(n)\ll n$ 
        \item[(C2.2.2)] $K_{0,2}(n) \ll 1$, $K_{1,2}(n) \ll 1$, and $K_{2,2}(n) \ll 1$.
    \end{itemize}
    \item[(C2.3)] $K_{0,1}(n)^2K_{0,2}(n) \ll n^2$, $K_{0,1}(n)^3K_{1,2}(n) \ll n^3$, and $K_{0,1}(n)^4K_{2,2}(n) \ll n^4$.
\end{enumerate}
\end{theorem}

We will begin by arguing that it is sufficient to prove that a slightly simplified set of conditions implies that $\lim_{n\to \infty} \P(\delta_{R_n} = 0)=1.$
First assume that conditions (C2.1), (C2.2.2), and (C2.3) hold.  
Condition (C2.2.2), combined with the fact that each $M_{i,j}(n)$ has a binomial distribution with mean $K_{i,j}(n)$,  yields
\[
    \lim_{n\to \infty}\P(M_{0,2}(n) = 0)=\lim_{n\to \infty}\P(M_{1,2}(n) = 0)=\lim_{n\to \infty}\P(M_{2,2}(n) = 0)=1.
\]
Hence, with probability approaching 1, the realized network only has edges in $E^{0,1}_n$ and $E^{1,1}_n$, and has a deficiency of zero by Lemma \ref{lemma:unary}.  Hence, the proof in this situation is done, and we can now simply assume that the conditions (C2.1), (C2.2.1), and (C2.3) are satisfied.

However, another slight simplification can take place.  Note that $K_{0,1}(n)\le n$ (since $|E^{0,1}_n| = n$), and if $K_{0,1}(n) \sim n$, then from condition (C2.3), we would have that condition (C2.2.2) is satisfies, which we already know implies the result.  Hence, we only need consider the case $K_{0,1}(n) \ll n$. For the other cases where there exist a subsequence along which $K_{0,1}(n)\sim n$ and another subsequence along which $K_{0,1}(n)\ll n$, we can apply the two corresponding arguments for the two subsequences, both of which when combined will still result in $\lim_{n\to \infty}\P(\delta_{R_n}=0)=1$.
Combining the above shows that Theorem \ref{main2} will be proved by showing that $\lim_{n\to \infty}\P(\delta_{R_n} = 0) = 1$ so long as the following conditions are satisfied:
\begin{enumerate}[leftmargin]
    \item[(C2.1*)] All $K_{i,j}(n) \ll n$.
    \item[(C2.3)\phantom{*}] $K_{0,1}(n)^2K_{0,2}(n) \ll n^2$, $K_{0,1}(n)^3K_{1,2}(n) \ll n^3$, and $K_{0,1}(n)^4K_{2,2}(n) \ll n^4$.
\end{enumerate}

Showing the above is the goal for the remainder of this section. In the first lemma, we construct some ``buffer'' functions, $Q_{i,j}(n)$ that are, asymptotically, between $K_{i,j}(n)$ and $n$, and also satisfy a version of condition (C2.3).  


\begin{lemma}\label{lemma:buffer}
If conditions (C2.1*) and (C2.3) hold, then there exists $Q_{0,1}(n)$, $Q_{0,2}(n)$, $Q_{1,1}(n)$, $Q_{1,2}(n)$, $Q_{2,2}(n)$ such that 
\begin{itemize}
\item $\lim_{n\to\infty}Q_{i,j}(n) >0$ for all $(i,j)$.
\item $K_{i,j}(n) \ll Q_{i,j}(n)\ll n$ for all $(i,j)$.
\item $Q_{0,1}(n)^2Q_{0,2}(n) \ll n^2$, $Q_{0,1}(n)^3Q_{1,2}(n) \ll n^3$, and $Q_{0,1}(n)^4Q_{2,2}(n) \ll n^4$,
\end{itemize}
\end{lemma}
\begin{proof}

We begin with $Q_{1,1}(n)$, which will be straightforward.  Set 
\[
    Q_{1,1}(n)=\max\{1,\sqrt{nK_{1,1}(n)}\}.
\]
From $K_{1,1}(n)\ll n$ in (C2.1$^*$) we have that $K_{1,1}(n) \ll Q_{1,1}(n) \ll n$ and $\lim_{n\to\infty}Q_{1,1}(n) >0$.

We turn to constructing $Q_{0,2}$.  In order to eventually convert the condition  $K_{0,1}(n)^2K_{0,2}(n)\ll n^2$ to the condition  $Q_{0,1}(n)^2Q_{0,2}(n)\ll n^2$, we will first construct a function $R_{0,1}(n)$, which satisfies $R_{0,1}(n)^2K_{0,2}(n)\ll n^2$.    We will then use $R_{0,1}(n)$ to build $Q_{0,2}(n)$ satisfying $R_{0,1}(n)^2Q_{0,2}(n)\ll n^2$.  After producing the pair $(R_{0,1}(n),Q_{0,1}(n)),$ we turn to producing similar pairs $(S_{0,1}(n),Q_{1,2}(n))$ and $(T_{0,1}(n),Q_{2,2}(n))$, each satisfying similar inequalities.  We will then define $Q_{0,1}(n)$ via the functions $R_{0,1}(n),S_{0,1}(n),T_{0,1}(n)$, and the proof will be complete.

Proceeding, we note that since $K_{0,1}(n)^2K_{0,2} \ll n^2$, we have $K_{0,1}(n)\ll \frac{n}{\sqrt{K_{0,2}(n)}}.$ By (C2.1$^*$), we have $K_{0,1}(n)\ll n$ as well. Let 
\[R_{0,1}(n)=\min\bigg\{\sqrt{K_{0,1}(n)\frac{n}{\sqrt{K_{0,2}(n)}}}, \sqrt{nK_{0,1}(n)} \bigg\}.\] 
The asymptotic inequalities above yield $K_{0,1}(n)\ll R_{0,1}(n)\ll n$ and $R_{0,1}(n)^2K_{0,2}(n)\ll n^2$. The final inequality implies $K_{0,2}(n)\ll\frac{n^2}{R_{0,1}(n)^2}$. We also have $K_{0,2}(n)\ll n$ from condition (C2.1$^*$). Finally, let
\[
Q_{0,2}(n)=\max\bigg\{1,\min\bigg\{ \sqrt{K_{0,2}(n)\frac{n^2}{R_{0,1}(n)^2}}, \sqrt{nK_{0,2}(n)} \bigg\}\bigg\}
\]
where the minimum is interpreted asymptotically as $n\to\infty$. Then we have $K_{0,2}(n)\ll Q_{0,2}(n)\ll n$ and $R_{0,1}(n)^2Q_{0,2}(n)\ll n^2$. 

We  mimic the above strategy and produce pairs of functions  $(S_{0,1}(n),Q_{1,2}(n))$ and $(T_{0,1}(n),Q_{2,2}(n))$ such that
\begin{itemize}
    \item $K_{0,1}(n)\ll S_{0,1}(n)\ll n$, $K_{1,2}(n)\ll Q_{1,2}(n) \ll n$, $\lim_{n\to\infty}Q_{1,2}(n) >0$ and $S_{0,1}(n)^3Q_{1,2}(n)\ll n^3$.
    \item $K_{0,1}(n)\ll T_{0,1}(n)\ll n$, $K_{2,2}(n)\ll Q_{2,2}(n) \ll n$, $\lim_{n\to\infty}Q_{2,2}(n) >0$ and $T_{0,1}(n)^4Q_{2,2}(n)\ll n^4$.
\end{itemize}
Finally, let
\[
Q_{0,1}(n)=\max\{1,\min\{R_{0,1}(n), S_{0,1}(n), T_{0,1}(n)\}\},
\]
where the minimum is interpreted asymptotically as $n\to\infty$. We now have all the $Q_{i,j}(n)$, and all the desired properties are straightforward to confirm.
\end{proof}

We turn to the main proof of Theorem \ref{main2}. The main proof utilizes some technical results, which will be proven in several lemmas after the main proof. 
\begin{proof}[Proof of Theorem \ref{main2}]
Assume that conditions (C2.1$^*$) and (C2.3) hold. We have
\begin{align}\label{eqjk23498u234}
    \P(\delta_{R_n}=0) = &\P(\delta_{R_n}=0,\cap_{i,j}\{M_{i,j}(n) \leq Q_{i,j}(n)\}) +\P(\delta_{R_n}=0,\cup_{i,j}\{M_{i,j}(n) > Q_{i,j}(n)\})
\end{align}
We will show that the second term goes to zero.  Since each $M_{i,j}(n)$ has a binomial distribution, we have
\begin{align*}
\P(M_{i,j}(n) > Q_{i,j}(n)) &= \P(M_{i,j}(n)-K_{i,j}(n)>Q_{i,j}(n)-K_{i,j}(n))\\
&\leq \frac{\var(M_{i,j}(n))}{(Q_{i,j}(n)-K_{i,j}(n))^2}\leq \frac{K_{i,j}(n)}{(Q_{i,j}(n)-K_{i,j}(n))^2}.
\end{align*}
Since $K_{i,j}(n)\ll Q_{i,j}(n)$ and $\lim_{n\to\infty}Q_{i,j}(n)>0$, we have $\lim_{n\to\infty}\P(M_{i,j}(n) > Q_{i,j}(n)) =0$ for all $(i,j)$. Thus
\begin{align}\label{eq8234nknj}
\lim_{n\to\infty} \P(\cup_{i,j}\{M_{i,j}(n) > Q_{i,j}(n)\})=0,
\end{align}
and consequently,
\begin{align}\label{eq32453249234}
\lim_{n\to\infty}\P(\delta_{R_n}=0,\cup_{i,j}\{M_{i,j}(n) > Q_{i,j}(n)\})=0.
\end{align}

Now we consider the first term in \eqref{eqjk23498u234}. We have
\begin{align*}
    \P(\delta_{R_n}=0&,\cap_{i,j}\{M_{i,j}(n) \leq Q_{i,j}(n)\})\nonumber\\
    &=\sum_{k_{i,j}(n)=0}^{Q_{i,j}(n)}\P(\delta_{R_n}=0|\cap_{i,j}\{M_{i,j}(n)=k_{i,j}(n)\})\P(\cap_{i,j}\{M_{i,j}(n)=k_{i,j}(n)\})\nonumber
\end{align*}
We will prove in Lemma \ref{lemma:main} below that
\begin{equation}\label{eq:345677}
\P(\delta_{R_n}=0|\cap_{i,j}\{M_{i,j}(n)=k_{i,j}(n)\}) \geq 1-C_3\frac{Q(n)}{n},
\end{equation}
where $Q(n)$ is a function satisfying $Q(n)\ll n$ and $C_3$ is independent from $n$ and $k_{i,j}(n)$. Thus
\begin{align}
    \P(\delta_{R_n}=0,\cap_{i,j}\{M_{i,j}(n) \leq Q_{i,j}(n)\})&\geq \bigg(1-C_3\frac{Q(n)}{n}\bigg)\sum_{k_{i,j}(n)=0}^{Q_{i,j}}\P(\cap_{i,j}\{M_{i,j}(n)=k_{i,j}(n)\})\nonumber\\
    &=\bigg(1-C_3\frac{Q(n)}{n}\bigg)\P(\cap_{i,j}\{M_{i,j}(n) \leq Q_{i,j}(n)\}).\label{eq23894n123}
\end{align}
Equation \eqref{eq8234nknj} gives us $\lim_{n\to\infty}\P(\cap_{i,j}\{M_{i,j}(n) \leq Q_{i,j}(n)\})=1$, thus from \eqref{eq23894n123} we have
\begin{align}\label{eq892348u9jhi23iu}
    \lim_{n\to\infty}\P(\delta_{R_n}=0&,\cap_{i,j}\{M_{i,j}(n) \leq Q_{i,j}(n)\}) = 1.
\end{align}
Combining \eqref{eqjk23498u234}, \eqref{eq32453249234}, and \eqref{eq892348u9jhi23iu} we have
\[
\lim_{n\to\infty}\P(\delta_{R_n}=0)=1.\qedhere
\]
\end{proof}

To complete this section, we will provide a series of lemmas, eventually  leading to Lemma \ref{lemma:main}, which yields the critical bound \eqref{eq:345677}
\[
\P(\delta_{R_n}=0|\cap_{i,j}\{M_{i,j}(n)=k_{i,j}(n)\}) \geq 1-C_3\frac{Q(n)}{n}.
\]
First we make an observation about the most probable number of species in realized reactions from each set $E_n^{i,j}$. Note that a reaction in the set $E_n^{0,2}$ can have either one or two distinct species appearing in it.  For example, we could have $\emptyset \leftrightharpoons 2S_1$, in which there is only one species, or we could have $\emptyset \leftrightharpoons S_1 + S_2$, in which there are two species.  Similarly, reactions from the set $E_n^{1,2}$ can have one, two, or three distinct species, and reactions from the set $E_n^{2,2}$ can have two, three, or four distinct species.  The following lemma states that when the number of realized reactions in each set is not too large, as quantified below, then, with probability approaching one as $n\to \infty$, the realized reactions from each set will consist of the maximal number of distinct species. 
\begin{lemma}\label{lemma:species}
Suppose conditions (C2.1*) and (C2.3) hold and that $Q_{i,j}(n)$ are defined as in Lemma \ref{lemma:buffer}.  Suppose further that $k_{i,j}(n) \leq Q_{i,j}(n)$. Let $A^{0,2}_n$, $A^{1,2}_n$, and $A^{2,2}_n$ be the events that the realized reactions in $E^{0,2}_n,E^{1,2}_n,E^{2,2}_n$  all have precisely  2,3, and 4 distinct species respectively. Let $A_n=A^{0,2}_n \cap A^{1,2}_n \cap A^{2,2}_n$ Then 
\[
\lim_{n\to \infty} \P(A_n|  \cap_{i,j}\{M_{i,j}(n)=k_{i,j}(n)\}) = 1.
\]
Moreover, we have the explicit bound
\[
\P(A_n|  \cap_{i,j}\{M_{i,j}(n)=k_{i,j}(n)\})\geq \bigg(1-\frac{2Q_{0,2}(n)}{n}\bigg) \bigg(1-\frac{4Q_{1,2}(n)}{n}\bigg) \bigg(1-\frac{8Q_{2,2}(n)}{n} \bigg).
\]
\end{lemma}
\begin{proof}
First, consider the reactions in $E^{0,2}_n$, which have the form $\emptyset \leftrightharpoons S_i+S_j$. These reactions have 2 species if and only if $i\neq j$. Recall that $|E^{0,2}_n|=n(n+1)/2$, and there are $n$ reactions of the form $2S_i$. Thus we have
\begin{align}
\P(A^{0,2}_n| &M_{0,2}(n)=k_{0,2}(n)) \nonumber \\
&= \bigg(1-\frac{n}{n(n+1)/2}\bigg)\bigg(1-\frac{n}{n(n+1)/2-1}\bigg)\cdots \bigg(1-\frac{n}{n(n+1)/2-k_{0,2}(n)+1}\bigg) \nonumber \\ 
&=\bigg(1-\frac{2}{n+1}\bigg)\bigg(1-\frac{2}{n+1-\frac{2}{n}}\bigg)\cdots \bigg(1-\frac{2}{n+1-\frac{2(k_{0,2}(n)-1)}{n}}\bigg)  \nonumber  \\ 
&\geq \bigg(1-\frac{2}{n}\bigg)^{k_{0,2}(n)} \nonumber \\
&\geq 1-\frac{2k_{0,2}(n)}{n}, \label{e189345jkh}
\end{align}
where the last inequality is due to Bernoulli's inequality. 

Next, consider the reactions in $E^{1,2}_n$.  These reactions have less than 3 species if it is either $S_i\leftrightharpoons S_i+S_j$ (where $i$ and $j$ are not necessarily different) or $S_i\to 2S_j$ (where $i\neq j$). It is straightforward to check that there are $n^2$ reactions of the former type, and there are $n(n-1)$ reactions of the latter type, both of which add up to $n(2n-1)$ reactions in $E^{1,2}_n$ with less than 3 species. Since $|E_n^{1,2}|=\frac{n^2(n+1)}{2}$ we have
\begin{align}
\P(A^{1,2}_n|&M_{1,2}(n)=k_{1,2}(n)) \nonumber \\
&= \bigg(1-\frac{n(2n-1)}{n^2(n+1)/2}\bigg)\bigg(1-\frac{n(2n-1)}{n^2(n+1)/2-1}\bigg)\cdots\bigg(1-\frac{n(2n-1)}{n^2(n+1)/2-k_{1,2}(n)+1}\bigg)  \nonumber \\ 
&\geq \bigg(1-\frac{4}{n}\bigg)^{k_{1,2}(n)} \nonumber  \\ 
&\geq 1-\frac{4k_{1,2}(n)}{n}, \label{eqn32491823oi}
\end{align}
where the first inequality here follows a similar argument to the first inequality in \eqref{e189345jkh}.

Finally, consider the reactions in $E^{2,2}_n$. These reactions have less than 4 species if they have the form $2S_i\leftrightharpoons 2S_j$, $2S_i\leftrightharpoons S_j+S_k$ (where $j\neq k$), or $S_i+S_j\leftrightharpoons S_i+S_k$ (where $i,j,k$ are pairwise different). It is straightforward to check that there are $\frac{n(n-1)}{2}$ reactions of the first type, $n(\frac{n(n+1)}{2}-n)$ reactions of the second type, and $(\frac{n(n+1)}{2}-n)(n-2)$ reactions of the third type. In total, there are $\frac{n(n-1)(2n-1)}{2}$ reactions in $E^{2,2}_n$ with less than 4 species. Since $|E_n^{2,2}|={\frac{n(n+1)}{2}\choose 2}=\frac{n(n+1)(n-1)(n+2)}{8}$, we have
\begin{align}
\P(A^{2,2}_n|&M_{2,2}(n)=k_{2,2}(n)) \nonumber\\
&=\bigg(1-\frac{\frac{n(n-1)(2n-1)}{2}}{\frac{n(n+1)(n-1)(n+2)}{8}}\bigg)\cdots\bigg(1-\frac{\frac{n(n-1)(2n-1)}{2}}{\frac{n(n+1)(n-1)(n+2)}{8}-k_{2,2}(n)+1}\bigg) \nonumber\\ 
&\geq \bigg(1-\frac{8}{n}\bigg)^{k_{2,2}(n)} \nonumber  \\ 
&\geq 1-\frac{8k_{2,2}(n)}{n},\label{eq123k123}
\end{align}
where the first inequality here follows a similar argument as the first inequality in \eqref{e189345jkh}.

From \eqref{e189345jkh},\eqref{eqn32491823oi},\eqref{eq123k123}, and independence, we have  
\begin{align*}
\P(A_n|&  \cap_{i,j}\{M_{i,j}(n)=k_{i,j}(n)\})] \nonumber \\
&\geq \bigg(1-\frac{2k_{0,2}(n)}{n}\bigg) \bigg(1-\frac{4k_{1,2}(n)}{n}\bigg) \bigg(1-\frac{8k_{2,2}(n)}{n} \bigg) \nonumber\\
&\geq \bigg(1-\frac{2Q_{0,2}(n)}{n}\bigg) \bigg(1-\frac{4Q_{1,2}(n)}{n}\bigg) \bigg(1-\frac{8Q_{2,2}(n)}{n} \bigg),
\end{align*}
and the limit follows.
\end{proof}

In our next major lemma, Lemma \ref{lemma:condition},  we require the notion of a minimally dependent set, which we define below.
\begin{definition}
We say a set of vectors is minimally dependent if it is linearly dependent and any of its proper subsets are linearly independent.
\end{definition}
We make a quick observation on minimally dependent set.
\begin{lemma}\label{lemma:row}
Let $M$ be a matrix whose columns $v_1,v_2,\dots,v_m$ are minimally dependent. Then $M$ has no row with only one non-zero entry. 
\end{lemma}
\begin{proof}
Since $v_1,\dots,v_m$ are dependent, there exist constants $\alpha_1,\dots,\alpha_m$, not all of which are zero, such that
\[
    \alpha_1v_1+\dots+\alpha_m v_m=0.
\]
Suppose by contradiction that $M$ has a row with only one non-zero entry, and suppose that entry belongs to the $i${th} column. Then this must imply $\alpha_i=0$.  However, this implies that
\[
    \sum_{j \ne i} \alpha_j v_j = 0,
\]
with not all $\alpha_j$ equaling zero.  This contradicts the set $\{v_i\}_{i = 1}^m$ being minimally dependent.
\end{proof}
An example related to minimal dependence in the context of reaction network is the network $\emptyset\leftrightharpoons S_1, \emptyset\leftrightharpoons S_2, \emptyset\leftrightharpoons S_3, \emptyset \leftrightharpoons S_1 + S_2$, whose reaction vectors are dependent, but not minimally dependent because the proper subset containing $\emptyset\leftrightharpoons S_1, \emptyset\leftrightharpoons S_2, \emptyset \leftrightharpoons S_1+S_2$ is dependent. In the next lemma, we will show that 
for a set of reaction vectors to be minimally dependent, there cannot be too many reactions from $E_n^{0,1}$, relative to the numbers from $E^{0,2}_n,E^{1,2}_n,E^{2,2}_n$.

\begin{lemma}\label{lemma:condition}
 Suppose a set $V$ with $i_1,i_2,i_3,i_4,i_5$ reaction vectors in $E^{0,1}_n,E^{1,1}_n,E^{0,2}_n,E^{1,2}_n,E^{2,2}_n$, respectively, is minimally dependent. Assume further that each of the $i_3,i_4$, and $i_5$ reactions from $E^{0,2}_n, E_n^{1,2},E_n^{2,2}$ have precisely 2,3, and 4 species, respectively, and that  $i_3+i_4+i_5>0$. Then we must have
\begin{align*}
    i_1\leq 2i_3+3i_4+4i_5.
\end{align*}
\end{lemma}
\begin{proof}
Consider a matrix $M$ whose first $i_1$ columns are the reaction vectors from $V \cap E^{0,1}_n$, the next $i_2$ columns are the reaction vectors from $V\cap E^{1,1}_n$, the next $i_3$ columns are the reaction vectors from $E^{0,2}_n$, etc.     Let $P$ be the sub-matrix consisting of the first $i_1+i_2$ columns of $M$ (so it is constructed by the reaction vectors from $V \cap E^{0,1}_n$ followed by the reaction vectors from $V \cap E^{1,1}_n$).  

Since $V$ is minimally dependent, Lemma \ref{lemma:row} tells us that $M$ has no row with only one non-zero entry. 
Let $z_{i_1+i_2}$ be the number of  rows of $P$ with exactly one entry.  By construction, the final $i_3 + i_4 + i_5$ columns of $M$ have at most
\[
    2i_3 + 3i_4 + 4i_5
\]
non-zero elements.  Therefore, we must have
\[
    z_{i_1+i_2} \le  2i_3 + 3i_4 + 4i_5,
\]
for otherwise there are not enough non-zero terms in the final $i_3 + i_4 + i_5$ columns to cover the rows of $P$ with a single element.
The remainder of the proof just consists of showing that 
\begin{equation}\label{eq:44444}
    i_1 \le z_{i_1 + i_2}.
\end{equation}
To show that the inequality \eqref{eq:44444} holds, we consider adding the column vectors sequentially, and make the following observations.
\begin{enumerate}
    \item The first $i_1$ columns of $M$ can, without loss of generality, be taken to be the canonical vectors $e_{1}, \dots, e_{i_1}$.  Note, therefore, that the sub-matrix consisting of the first $i_1$ columns of $M$ has exactly $i_1$ rows that have a single non-zero entry.
    
    \item The rank of the sub-matrix of $P$ consisting of the first $i_1 + k$ columns must be $i_1 +k$ for any $0\le k \le i_2$, for otherwise there is a dependence and $V$ would not be minimally dependent (here we are explicitly using that $i_3 + i_4 + i_5 > 0)$.  
    \item Consider the action of going from a sub-matrix of $P$ consisting of the first $i_1 +k$ columns to one consisting of the first $i_1 + k +1$ columns, for $k\le i_2 -1$.  Since each such sub-matrix is full rank (by the point made above), the addition of the next column vector in the construction  must have at least one element in a row that was previously all zeros.  
    \item Since each column vector being added has at most two elements, the number of rows with a single entry can never decrease.
\end{enumerate}
Hence, we have that the number of rows with precisely one non-zero entry at the end of the construction, $z_{i_1 + i_2}$ must be at least as large as the number at the beginning of the construction, $i_1$, and we are done. 
\end{proof}

Finally, we present the main lemma, giving the bound needed for Theorem \ref{main2}. Note that a positive deficiency must imply the existence of a minimally independent set. Thus the main approach of the proof revolves around summing over the probabilities of each certain set of reaction vectors being minimally dependent. The constraint in Lemma \ref{lemma:condition} will play a critical role in this approach.

\begin{lemma}\label{lemma:main}
Suppose conditions (C2.1*) and (C2.3) hold and that $Q_{i,j}(n)$ are  as in Lemma \ref{lemma:buffer}.  Suppose further that $k_{i,j}(n) \leq Q_{i,j}(n)$ for each relevant pair $(i,j)$.
Then
\[
\P(\delta_{R_n}=0 | \cap_{i,j}\{M_{i,j}(n)=k_{i,j}(n)\}) \geq 1-C_3\frac{Q(n)}{n},
\]
where $Q(n)$ is a function satisfying $Q(n)\ll n$ and $C_3$ is independent from $n$ and $k_{i,j}(n)$.
\end{lemma}
\begin{proof}
We have
\begin{align}
\P(&\delta_{R_n}=0|\cap_{i,j}\{M_{i,j}(n)=k_{i,j}(n)\})\nonumber \\
&\geq \P(\delta_{R_n}=0|A_n, \cap_{i,j}\{M_{i,j}(n)=k_{i,j}(n)\})\P(A_n| \cap_{i,j}\{M_{i,j}(n)=k_{i,j}(n)\})\nonumber \\
&=(1-\P(\delta_{R_n}>0|A_n, \cap_{i,j}\{M_{i,j}(n)=k_{i,j}(n)\}))\P(A_n| \cap_{i,j}\{M_{i,j}(n)=k_{i,j}(n)\}).\label{eq234jl23j4l234}
\end{align}
From Lemma \ref{lemma:unary} and Lemma \ref{lemma:independent}, the event $\delta_{R_n}>0$ must imply there exists a minimally dependent set which consists of at least one reaction from $E^{0,2}_n$, $E^{1,2}_n$, or $E^{2,2}_n$. Let $I=(i_1,i_2,i_3,i_4,i_5)$ be a multi-index. Let $K_n=(k_{0,1}(n),k_{1,1}(n),k_{0,2}(n),k_{1,2}(n),k_{2,2}(n))$. For convenience, we write $I\leq K_n$ to represent $i_1\leq k_{0,1}(n), \dots, i_5\leq k_{2,2}(n)$. Then we have
\begin{align}
&\P(\delta_{R_n}>0|A_n, \cap_{i,j}\{M_{i,j}(n)=k_{i,j}(n)\})\nonumber \\
&\leq \sum_{\substack{I\leq K_n\\i_3+i_4+i_5>0\\i_1\leq 2i_3+3i_4+4i_5}}{k_{0,1}(n)\choose i_1}{k_{1,1}(n)\choose i_2}{k_{0,2}(n)\choose i_3}{k_{1,2}(n)\choose i_4}{k_{2,2}(n)\choose i_5}\P(B_I|A_n, \cap_{i,j}\{M_{i,j}(n)=k_{i,j}(n)\}),\label{eq239804jkl}
\end{align}
where $B_I$ is the event that a set with $i_1,i_2,i_3,i_4,i_5$ realized reactions from $E^{0,1}_n$,  $E^{1,1}_n$, $E^{0,2}_n$, $E^{1,2}_n$, $E^{2,2}_n$, which also satisfy $A_n$, is minimally dependent. Note that the constraint $i_1\leq 2i_3+3i_4+4i_5$ comes from Lemma \ref{lemma:condition}.

Now we fix an index $I=(i_1,i_2,i_3,i_4,i_5)$ and we fix a particular minimally dependent reaction set $V_I$ with $i_1,i_2,i_3,i_4,i_5$ reactions in $E^{0,1}_n,E^{1,1}_n,E^{0,2}_n,E^{1,2}_n,E^{2,2}_n$. Let $M_I$ be the matrix whose columns are  reaction vectors in $V_I$.  Next, we notice that the total number of non-zero entries in $M_I$ is $i_1+2i_2+2i_3+3i_4+4i_5$. Since each non-zero row in $M_I$ must have at least two non-zero entries, the number of non-zero rows is at most $\ell:=\lfloor \frac{i_1+2i_2+2i_3+3i_4+4i_5}{2}\rfloor$. 

There are ${n\choose \ell}$ ways to choose $\ell$ non-zero rows from $n$ rows. Fix a set of $\ell$ rows to be non-zero rows. We have the probability that all $i_1$ reaction vectors in $E^{0,1}_n$ have non-zero entry among these $\ell$ rows is 
\[
\frac{\ell}{n}\frac{\ell-1}{n-1}\cdots\frac{\ell-i_1+1}{n-i_1+1}\leq \bigg(\frac{\ell}{n}\bigg)^{i_1}.
\]
The probability that all $i_2$ reactions vectors in $E^{1,1}_n$ have non-zero entry among these $\ell$ rows is
\[
\frac{{\ell\choose 2}}{{n\choose 2}}\frac{{\ell\choose 2}-1}{{n\choose 2}-1}\cdots\frac{{\ell\choose 2}-i_2+1}{{n\choose 2}-i_2+1}\leq \bigg(\frac{\ell}{n}\bigg)^{2i_2}.
\]
Using similar arguments, we have 
\begin{align}
\P(B_I|A_n, \cap_{i,j}\{M_{i,j}(n)=k_{i,j}(n)\})&\leq  {n\choose \ell}\bigg(\frac{\ell}{n}\bigg)^{i_1+2i_2+2i_3+3i_4+4i_5}\nonumber \\
&\leq \frac{n^{\ell}}{\ell!}\bigg(\frac{\ell}{n}\bigg)^{\ell+\frac{i_1+2i_2+2i_3+3i_4+4i_5}{2}} \nonumber \\
&\leq \frac{n^\ell}{\ell^\ell e^{-\ell}}\bigg(\frac{\ell}{n}\bigg)^{\ell+\frac{i_1+2i_2+2i_3+3i_4+4i_5}{2}} \nonumber\\
&\leq \bigg(\frac{e\ell}{n}\bigg)^{\frac{i_1+2i_2+2i_3+3i_4+4i_5}{2}}, \label{eq89024klnjk}
\end{align}
where the third inequality is due to the inequality $x!\geq x^x e^{-x}$. Combining \eqref{eq239804jkl} and \eqref{eq89024klnjk}, we have
\begin{align*}
&\P(\delta_{R_n}>0|A_n, \cap_{i,j}\{M_{i,j}(n)=k_{i,j}(n)\}) \\
&\leq \sum_{\substack{I\leq K_n\\i_3+i_4+i_5>0\\i_1\leq 2i_3+3i_4+4i_5}}{k_{0,1}(n)\choose i_1}{k_{1,1}(n)\choose i_2}\cdots{k_{2,2}(n)\choose i_5}\bigg(\frac{e\ell}{n}\bigg)^{\frac{i_1+2i_2+2i_3+3i_4+4i_5}{2}} \\
&\leq \sum_{\substack{I\leq K_n\\i_3+i_4+i_5>0\\i_1\leq 2i_3+3i_4+4i_5}}\bigg(\frac{ek_{0,1}(n)}{i_1}\bigg)^{i_1}\cdots\bigg(\frac{ek_{2,2}(n)}{i_5}\bigg)^{i_5}\bigg(\frac{e\ell}{n}\bigg)^{\frac{i_1+2i_2+2i_3+3i_4+4i_5}{2}} \\
&\leq \sum_{\substack{I\leq K_n\\i_3+i_4+i_5>0\\i_1\leq 2i_3+3i_4+4i_5}}\bigg(\frac{(5ek_{0,1}(n))^{i_1}\cdots(5ek_{2,2}(n))^{i_5}}{(i_1+i_2+i_3+i_4+i_5)^{i_1+i_2+i_3+i_4+i_5}}\bigg)\bigg(\frac{e\ell}{n}\bigg)^{\frac{i_1+2i_2+2i_3+3i_4+4i_5}{2}} \\
&\leq \sum_{\substack{I\leq K_n\\i_3+i_4+i_5>0\\i_1\leq 2i_3+3i_4+4i_5}}\bigg(\frac{(5eQ_{0,1}(n))^{i_1}\cdots(5eQ_{2,2}(n))^{i_5}}{(i_1+i_2+i_3+i_4+i_5)^{i_1+i_2+i_3+i_4+i_5}}\bigg)\bigg(\frac{e\ell}{n}\bigg)^{\frac{i_1+2i_2+2i_3+3i_4+4i_5}{2}},
\end{align*}
where the second inequality is again due to $x!\geq x^xe^{-x}$ and the third inequality is due to Corollary \ref{cor:inequality}. Since $\ell=\lfloor\frac{i_1+2i_2+2i_3+3i_4+4i_5}{2}\rfloor\leq 2(i_1+i_2+i_3+i_4+i_5)$, we have
\begin{align}
&\P(\delta_{R_n}>0|A_n, \cap_{i,j}\{M_{i,j}(n)=k_{i,j}(n)\}) \nonumber\\ 
&\leq \sum_{\substack{I\leq K_n\\i_3+i_4+i_5>0\\i_1\leq 2i_3+3i_4+4i_5}}\frac{(5eQ_{0,1}(n))^{i_1}\cdots(5eQ_{2,2}(n))^{i_5}(i_1+i_2+i_3+i_4+i_5)^{\ell-(i_1+i_2+i_3+i_4+i_5)}}{((2e)^{-1}n)^{\frac{i_1+2i_2+2i_3+3i_4+4i_5}{2}}}\nonumber \\
&\leq \sum_{\substack{I\leq K_n\\i_3+i_4+i_5>0\\i_1\leq 2i_3+3i_4+4i_5}}\frac{(5eQ_{0,1}(n))^{i_1}\cdots(5eQ_{2,2}(n))^{i_5}(i_1+i_2+i_3+i_4+i_5)^{-\frac{i_1}{2}+\frac{i_4}{2}+i_5}}{((2e)^{-1}n)^{\frac{i_1+2i_2+2i_3+3i_4+4i_5}{2}}}\nonumber\\
&= S_n+T_n,\label{eq234234234234234}
\end{align}
where
\[
S_n=\sum_{\substack{I\leq K_n\\i_3+i_4+i_5>0\\i_1\leq 2i_3+3i_4+4i_5\\i_1\leq i_4+2i_5}}\frac{(5eQ_{0,1}(n))^{i_1}\cdots(5eQ_{2,2}(n))^{i_5}(i_1+i_2+i_3+i_4+i_5)^{-\frac{i_1}{2}+\frac{i_4}{2}+i_5}}{((2e)^{-1}n)^{\frac{i_1+2i_2+2i_3+3i_4+4i_5}{2}}},
\]
consists of the terms with positive exponent for $i_1+i_2+i_3+i_4+i_5$ and
\[
T_n=\sum_{\substack{I\leq K_n\\i_3+i_4+i_5>0\\i_1\leq 2i_3+3i_4+4i_5\\i_1> i_4+2i_5}}\frac{(5eQ_{0,1}(n))^{i_1}\cdots(5eQ_{2,2}(n))^{i_5}(i_1+i_2+i_3+i_4+i_5)^{-\frac{i_1}{2}+\frac{i_4}{2}+i_5}}{((2e)^{-1}n)^{\frac{i_1+2i_2+2i_3+3i_4+4i_5}{2}}}
\]
consists of the terms with negative exponent for $i_1+i_2+i_3+i_4+i_5$.

We first deal with $T_n$, which is the more difficult term to bound. Notice that the exponent $-\frac{i_1}{2}+\frac{i_4}{2}+i_5<0$. Therefore we have
\begin{align}
    T_n &\leq \sum_{\substack{I\leq K_n\\i_3+i_4+i_5>0\\i_1\leq 2i_3+3i_4+4i_5\\i_1> i_4+2i_5}}\frac{(5eQ_{0,1}(n))^{i_1}\cdots(5eQ_{2,2}(n))^{i_5}}{((2e)^{-1}n)^{\frac{i_1+2i_2+2i_3+3i_4+4i_5}{2}}}\nonumber\\
    &= \sum_{\substack{I\leq K_n\\i_3+i_4+i_5>0\\i_1\leq 2i_3+3i_4+4i_5\\i_1> i_4+2i_5}}\frac{Q_{0,1}(n)^{i_1}\cdots Q_{2,2}(n)^{i_5}(5e)^{i_1+i_2+i_3+i_4+i_5}}{((2e)^{-1}n)^{\frac{i_1+2i_2+2i_3+3i_4+4i_5}{2}}}\nonumber\\
    &\leq \sum_{\substack{I\leq K_n\\i_3+i_4+i_5>0\\i_1\leq 2i_3+3i_4+4i_5\\i_1> i_4+2i_5}}\frac{Q_{0,1}(n)^{i_1}\cdots Q_{2,2}(n)^{i_5}}{((50e^3)^{-1}n)^{\frac{i_1+2i_2+2i_3+3i_4+4i_5}{2}}}, \label{eq391284oi}
\end{align}
where the last inequality is due to the fact that $i_1+i_2+i_3+i_4+i_5 \leq 2\frac{i_1+2i_2+2i_3+3i_4+4i_5}{2}$. 
Let 
\begin{align}
Q(n)=\max\{Q_{i,j}(n), Q_{0,1}(n)Q_{0,2}(n)^{1/2}, Q_{0,1}(n)Q_{1,2}(n)^{1/3}, Q_{0,1}(n)Q_{2,2}(n)^{1/4} \},
\end{align}
where the maximum is interpreted asymptotically as $n\to\infty$. From the way we construct $Q_{i,j}(n)$ in Lemma \ref{lemma:buffer} we have $Q(n)\ll n$. Next, we split $Q_{0,1}(n)^{i_1}$ into the product of three terms and distribute them into $Q_{0,2}(n),Q_{1,2}(n)$, and $Q_{2,2}(n)$. We have 
\[
Q_{0,1}(n)^{i_1\frac{2i_3}{2i_3+3i_4+4i_5}}Q_{0,2}(n)^{\frac{i_1}{2}\frac{2i_3}{2i_3+3i_4+4i_5}}\leq Q(n)^{i_1\frac{2i_3}{2i_3+3i_4+4i_5}},
\]
\[
Q_{0,1}(n)^{i_1\frac{3i_4}{2i_3+3i_4+4i_5}}Q_{1,2}(n)^{\frac{i_1}{3}\frac{3i_4}{2i_3+3i_4+4i_5}}\leq Q(n)^{i_1\frac{3i_4}{2i_3+3i_4+4i_5}},
\]
and
\[
Q_{0,1}(n)^{i_1\frac{4i_5}{2i_3+3i_4+4i_5}}Q_{2,2}(n)^{\frac{i_1}{4}\frac{4i_5}{2i_3+3i_4+4i_5}}\leq Q(n)^{i_1\frac{4i_5}{2i_3+3i_4+4i_5}}.
\]
Multiplying these inequalities together, we have
\[
Q_{0,1}(n)^{i_1}Q_{0,2}(n)^{i_1\frac{i_3}{2i_3+3i_4+4i_5}}Q_{1,2}(n)^{i_1\frac{i_4}{2i_3+3i_4+4i_5}}Q_{2,2}(n)^{i_1\frac{i_5}{2i_3+3i_4+4i_5}} \leq Q(n)^{i_1}.
\]
Note that in \eqref{eq391284oi}, $Q_{0,2}(n)$ has an exponent of $i_3$. Notice further that $i_1\frac{i_3}{2i_3+3i_4+4i_5} \leq i_3$, since $i_1\leq 2i_3+3i_4+4i_5$. Thus we have
\[
Q_{0,2}(n)^{i_3-i_1\frac{i_3}{2i_3+3i_4+4i_5}} \leq Q(n)^{i_3-i_1\frac{i_3}{2i_3+3i_4+4i_5}}.
\]
Similarly, we have
\[
Q_{1,2}(n)^{i_4-i_1\frac{i_4}{2i_3+3i_4+4i_5}} \leq Q(n)^{i_4-i_1\frac{i_4}{2i_3+3i_4+4i_5}}
\]
and
\[
Q_{2,2}(n)^{i_5-i_1\frac{i_5}{2i_3+3i_4+4i_5}} \leq Q(n)^{i_5-i_1\frac{i_5}{2i_3+3i_4+4i_5}}.
\]
Therefore we have 
\begin{align}
Q_{0,1}(n)^{i_1}\cdots Q_{2,2}(n)^{i_5} &\leq Q(n)^{i_1+i_2}Q(n)^{i_3+i_4+i_5-i_1\frac{i_3+i_4+i_5}{2i_3+3i_4+4i_5}}\nonumber\\
& = Q(n)^{i_1+i_2+i_3+i_4+i_5-i_1\frac{i_3+i_4+i_5}{2i_3+3i_4+4i_5}}.\label{e189234njkjo}
\end{align}
Note that $i_1\leq 2i_3+3i_4+4i_5$, thus
\begin{align}
i_1+i_2+i_3+i_4+i_5-i_1\frac{i_3+i_4+i_5}{2i_3+3i_4+4i_5}\leq \frac{i_1+2i_2+2i_3+3i_4+4i_5}{2}, \label{eq1ji2h43iuuoi}
\end{align}
where the inequality above can be verified in a straightforward manner. Combining \eqref{eq391284oi},\eqref{e189234njkjo}, and \eqref{eq1ji2h43iuuoi}, and noting that $i_3+i_4+i_5>0$, we have 
\begin{align}
T_n &\leq \sum_{\substack{I\leq K_n\\i_3+i_4+i_5>0\\i_1\leq 2i_3+3i_4+4i_5\\i_1> i_4+2i_5}} \bigg(\frac{Q(n)}{(50e^3)^{-1}n}\bigg)^{\frac{i_1+2i_2+2i_3+3i_4+4i_5}{2}}\nonumber \\
&\leq \frac{Q(n)}{(50e^3)^{-1}n}\sum_{i_1=0}^\infty \bigg(\frac{Q(n)}{(50e^3)^{-1}n}\bigg)^{i_1/2} \cdots\sum_{i_5=0}^\infty \bigg(\frac{Q(n)}{(50e^3)^{-1}n}\bigg)^{2i_5}\nonumber \\
&\leq C_1\frac{Q(n)}{n}\label{eqo23u49},
\end{align}
where the second inequality is due to the fact that $i_3+i_4+i_5>0$. Since each sum on the right hand side is bounded by $2$ for $n$ large enough, the constant $C_1$ is independent from $n$ and $k_{i,j}(n)$.

Next we consider $S_n$. Recall that $i_1\leq k_{0,1}(n) \leq Q_{0,1}(n), \dots, i_5\leq k_{2,2}(n)\leq Q_{2,2}(n)$, implying $i_1,\dots,i_5\leq Q(n)$. Therefore we have
\begin{align}
S_n &\leq \sum_{\substack{I\leq K_n\\i_3+i_4+i_5>0\\i_1\leq 2i_3+3i_4+4i_5\\i_1\leq i_4+2i_5}}\frac{(5eQ(n))^{i_1+i_2+i_3+i_4+i_5-\frac{i_1}{2}+\frac{i_4}{2}+i_5}}{((2e)^{-1}n)^{\frac{i_1+2i_2+2i_3+3i_4+4i_5}{2}}} \nonumber\\
&=\sum_{\substack{I\leq K_n\\i_3+i_4+i_5>0\\i_1\leq 2i_3+3i_4+4i_5\\i_1\leq i_4+2i_5}}\frac{(5eQ(n))^{\frac{i_1+2i_2+2i_3+3i_4+4i_5}{2}}}{((2e)^{-1}n)^{\frac{i_1+2i_2+2i_3+3i_4+4i_5}{2}}} \nonumber\\
&\leq \frac{5eQ(n)}{(2e)^{-1}n}\sum_{i_1=0}^\infty \bigg(\frac{5eQ(n)}{(2e)^{-1}n}\bigg)^{i_1/2} \cdots\sum_{i_5=0}^\infty \bigg(\frac{5eQ(n)}{(2e)^{-1}n}\bigg)^{2i_5}\nonumber\\
&\leq C_2\frac{Q(n)}{n},\label{eq123krwe}
\end{align}
where $C_2$ is independent from $n$ and $k_{i,j}(n)$. From \eqref{eq234234234234234}, \eqref{eqo23u49}, \eqref{eq123krwe}, we have
\begin{align}\label{eq23894jhi234}
\P(\delta_{R_n}>0|A_n, \cap_{i,j}\{M_{i,j}(n)=k_{i,j}(n)\}) \leq C_1\frac{Q(n)}{n}+C_2\frac{Q(n)}{n}.
\end{align}
From Lemma \ref{lemma:species} and the fact that $Q_{i,j}(n)\leq Q(n)$, we have
\begin{align}\label{eq9234njjk23j423o23C}
\P(A_n| \cap_{i,j}\{M_{i,j}(n)=k_{i,j}(n)\}) \geq \bigg(1-\frac{2Q(n)}{n}\bigg) \bigg(1-\frac{4Q(n)}{n}\bigg) \bigg(1-\frac{8Q(n)}{n} \bigg).
\end{align}
Plugging \eqref{eq23894jhi234} and \eqref{eq9234njjk23j423o23C} into \eqref{eq234jl23j4l234} yields
\begin{align}
\P(&\delta_{R_n}=0|\cap_{i,j}\{M_{i,j}(n)=k_{i,j}(n)\})\nonumber\\
&\geq \bigg(1-C_1\frac{Q(n)}{n}-C_2\frac{Q(n)}{n}\bigg) \bigg(1-\frac{2Q(n)}{n}\bigg) \bigg(1-\frac{4Q(n)}{n}\bigg) \bigg(1-\frac{8Q(n)}{n} \bigg)\nonumber\\
&\geq 1- C_3\frac{Q(n)}{n},
\end{align}
where the last inequality is obtained from repeatedly applying $(1-a)(1-b)\geq 1-a-b$ (where $a,b\geq 0$). Clearly we must have $C_3$ independent from $n$ and $k_{i,j}(n)$.
\end{proof}

\section{The threshold function for deficiency zero}\label{sec6}
In this section, we  provide an algorithm to find the threshold function $r(n)$ for deficiency zero for a given set of $\{\alpha_{i,j}\}$.  Specifically, $r(n)$ will satisfy
\begin{enumerate}
    \item $\lim_{n\to\infty}\P(\delta_{R_n}=0)= 0$ for $\lim_{n\to\infty}\frac{p_n}{r(n)}=\infty$, and
    \item $\lim_{n\to\infty}\P(\delta_{R_n}=0)= 1$ for $\lim_{n\to\infty}\frac{p_n}{r(n)}=0.$
\end{enumerate}

From Remark 5, we have $K_{i,j}(n)\sim n^{i+j}n^{\alpha_{i,j}}p_n = n^{i+j +\alpha_{i,j}}p_n$.  Moreover, from Section \ref{sec5}, we have  sets of conditions on the $K_{i,j}(n)$ that determine when a network does or does not have a deficiency of zero.  Combining these yields the following theorem.  In the theorem below, note that the equations 1-3 correspond to condition (C1.1) (and (C2.1)), the equations 4-7 correspond to condition (C1.2) (and (C2.2)), and the equations 8-10 correspond to condition (C1.3) (and (C2.3)).

\begin{theorem}\label{theorem:pn}
Given a set of parameters $\{\alpha_{i,j}\}$, consider the following systems where we solve for $\{r_i(n)\}$
\begin{enumerate}
    \item $n^{2+\alpha_{0,2}} r_1(n) = n$.
    \item $n^{3+\alpha_{1,2}} r_2(n) = n$.
    \item $n^4r_3(n)= n$.
    \item $n^{2+\alpha_{1,1}} r_4(n) = n$.
    \item  $n^{2+\alpha_{0,2}} r_5(n)= 1$.
    \item  $n^{3+\alpha_{1,2}} r_6(n)= 1$.
    \item  $n^4r_7(n) = 1$.
    \item $n^{4+2\alpha_{0,1}+\alpha_{0,2}}r_8(n)^3 = n^2$.
    \item $n^{6+3\alpha_{0,1}+\alpha_{1,2}}r_9(n)^4 = n^3$.
    \item $n^{8+4\alpha_{0,1}}r_{10}(n)^5 = n^4$.
\end{enumerate}
Then the threshold function is 
\[r(n)=\min\{r_1(n),r_2(n),r_3(n),\max\{r_4(n),\min\{r_5(n),r_6(n),r_7(n)\}\},r_8(n),r_9(n),r_{10}(n)\},
\]
where the maximum and minimum are interpreted asymptotically as $n\to\infty$. 
\end{theorem}

\begin{proof}
If $\lim_{n\to\infty}\frac{p_n}{r(n)}=\infty$, then it is easy to show that at least one condition in Theorem \ref{main1} is satisfied. Similarly, if $\lim_{n\to\infty}\frac{p_n}{r(n)}=0$, then all conditions in Theorem \ref{main2} are satisfied.
\end{proof}

\begin{example}[A closed system with $\alpha_{0,1}=\alpha_{0,2}=0, \alpha_{1,1}=2, \alpha_{1,2}=1$]
In this case, we have $K_{0,1}(n)\sim np_n$, $K_{0,2}(n)\sim n^2p_n$, $K_{1,1}(n)\sim K_{1,2}(n)\sim K_{2,2}(n)\sim n^4p_n$. Using Theorem \ref{theorem:pn} yields
\[
r(n)=\frac{1}{n^3},
\]
which is the same threshold as in the base case in \cite{prevalence}. \hfill $\triangle$
\end{example}

\begin{example}[An open system with $\alpha_{0,1}=3$, $\alpha_{1,1}=\alpha_{0,2}=2$, $\alpha_{1,2}=1$]. In this case, we have $K_{i,j}(n)\sim n^4p_n$ for all $(i,j)$. Using Theorem \ref{theorem:pn} yields
\[
r(n)=\frac{1}{n^{10/3}},
\]
which is a lower threshold than the previous case with a closed system. Intuitively, the inflow and outflow reactions make it easier to break deficiency zero of a reaction network. \hfill $\triangle$
\end{example}

\vspace{.15in}

\noindent {\Large \textbf{Acknowledgements}}

\vspace{.15in}

\noindent We gratefully acknowledge support  via the Army Research Office through grant W911NF-18-1-0324, and via the William F. Vilas Trust Estate. 

\appendix
\section{Appendix}
The following lemmas have been used in the manuscript.  Their proofs are added for completenes.
\begin{lemma}\label{lemma:binomial}
Let $X\sim B(n,p)$. Then we have
\[
\E\bigg[ \frac{1}{X+1}\bigg] \leq \frac{1}{np}, 
\quad \text{and} \quad 
\E\bigg[ \frac{1}{(X+1)(X+2)}\bigg] \leq \frac{1}{(np)^2}.
\]
\end{lemma}
\begin{proof}
We have
\begin{align*}
\E\bigg[\frac{1}{X+1}\bigg] &= \sum_{i=0}^n \frac{1}{i+1}{n\choose i}p^i(1-p)^{n-i} \nonumber =\sum_{i=0}^n  \frac{n!}{(i+1)!(n-i)!}p^i(1-p)^{n-i} \nonumber\\
&=\frac{1}{n+1}\frac{1}{p}\sum_{i=0}^n {n+1\choose i+1}p^{i+1}(1-p)^{n-i} \leq \frac{1}{np}(p+1-p)^{n+1} \leq \frac{1}{np}.
\end{align*}
Similarly, we have
\begin{align*}
\E\bigg[\frac{1}{(X+1)(X+2)}\bigg] &= \sum_{i=0}^n \frac{1}{(i+1)(i+2)}{n\choose i}p^i(1-p)^{n-i} =\sum_{i=0}^n  \frac{n!}{(i+2)!(n-i)!}p^i(1-p)^{n-i} \nonumber\\
&=\frac{1}{(n+1)(n+2)}\frac{1}{p^2}\sum_{i=0}^n {n+2\choose i+2}p^{i+2}(1-p)^{n-i} \nonumber\\
&\leq \frac{1}{(np)^2}(p+1-p)^{n+2} \leq \frac{1}{(np)^2}.
\end{align*}
\end{proof}
\begin{lemma}\label{lemma:inequality}
Let $x,y\in\R_{\geq 0}$. Then we have
\[
(2x)^x(2y)^y \geq (x+y)^{x+y}
\]
\end{lemma}
\begin{proof}
Clearly the inequality holds when either $x=0$ or $y=0$ or both. Suppose $x>0$ and $y>0$. We have
\begin{align*}
(2x)^x(2y)^y \geq (x+y)^{x+y}
\iff 2^{x+y}\bigg(\frac{x}{y}\bigg)^x\geq \bigg(1+\frac{x}{y}\bigg)^{x+y} \iff 2^{1+\frac{x}{y}}\bigg(\frac{x}{y}\bigg)^{x/y}\geq \bigg(1+\frac{x}{y}\bigg)^{1+x/y}.
\end{align*}
Thus the inequality holds if we have $2^{1+t}t^t \geq (1+t)^{1+t}$, or $(1+t)\ln(2) + t\ln(t) \geq (1+t)\ln(1+t)$ for $t>0$. Let 
\[
f(t)=(1+t)\ln(2) + t\ln(t) - (1+t)\ln(1+t).
\]
A quick calculation shows $f'(t)=\ln(2t)-\ln(1+t)$, and $f(t)$ has a global minimum at $t=1$. Thus $f(t)\geq f(1)=0$, which concludes the proof of the Lemma.
\end{proof}
\begin{corollary}\label{cor:inequality}
Let $x_1,x_2,\dots,x_n\in \R_{\geq 0}$, then we have
\begin{align}\label{eq234kl09}
\prod_{i=1}^n(nx_i)^{x_i} \geq \bigg(\sum_{i=1}^n x_i\bigg)^{\sum_{i=1}^n x_i}.
\end{align}
\end{corollary}
\begin{proof}
We will prove the corollary by induction. Clearly \eqref{eq234kl09} holds for $n=1$. Lemma \ref{lemma:inequality} shows that \eqref{eq234kl09} holds for $n=2$. Suppose \eqref{eq234kl09} holds for $n=k$. It suffices to show that \eqref{eq234kl09} holds for $n=2k$ and $n=k-1$.

First, we will show that \eqref{eq234kl09} holds for $n=2k$. Applying the inductive hypothesis for the $n=k$ terms $x_1,\dots,x_k$ and the $n=k$ terms  $x_{k+1},\dots,x_{2k}$, and then applying Lemma \ref{lemma:inequality} yields
\begin{align*}
\prod_{i=1}^{2k}(2kx_i)^{2x_i} \geq \bigg(\sum_{i=1}^k 2x_i\bigg)^{\sum_{i=1}^k 2x_i}\bigg(\sum_{i=k+1}^{2k} 2x_i\bigg)^{\sum_{i=k+1}^{2k} 2x_i}\geq \bigg(\sum_{i=1}^{2k}x_i\bigg)^{2\sum_{i=1}^{2k}x_i}.
\end{align*}
Taking square root of the inequality above gives us the case $n=2k$.

Next, we will show that \eqref{eq234kl09} holds for $n=k-1$. Applying the induction hypothesis for the $n=k$ terms $x_1,\dots,x_{k-1},\frac{1}{k-1}\sum_{i=1}^{k-1}x_i$, we have
\begin{align*}
&\prod_{i=1}^{k-1}(kx_i)^{x_i}\bigg(\frac{k}{k-1}\sum_{i=1}^{k-1}x_i\bigg)^{\frac{1}{k-1}\sum_{i=1}^{k-1}x_i} \geq \bigg(\sum_{i=1}^{k-1} x_i+\frac{1}{k-1}\sum_{i=1}^{k-1}x_i\bigg)^{\sum_{i=1}^{k-1} x_i+\frac{1}{k-1}\sum_{i=1}^{k-1}x_i}\\
\Rightarrow &\prod_{i=1}^{k-1}(kx_i)^{x_i}\bigg(\frac{k}{k-1}\sum_{i=1}^{k-1}x_i\bigg)^{\frac{1}{k-1}\sum_{i=1}^{k-1}x_i}  \geq \bigg(\frac{k}{k-1}\sum_{i=1}^{k-1}x_i\bigg)^{\frac{k}{k-1}\sum_{i=1}^{k-1} x_i}\\
\Rightarrow &\prod_{i=1}^{k-1}(kx_i)^{x_i} \geq \bigg(\frac{k}{k-1}\sum_{i=1}^{k-1}x_i\bigg)^{\sum_{i=1}^{k-1} x_i} \Rightarrow \prod_{i=1}^{k-1}((k-1)x_i)^{x_i} \geq \bigg(\sum_{i=1}^{k-1}x_i\bigg)^{\sum_{i=1}^{k-1} x_i}.
\end{align*}
Thus we have show that \eqref{eq234kl09} holds for $n=k-1$, which concludes the proof of the Corollary.
\end{proof}

\bibliographystyle{plain}
	\bibliography{bib}
\end{document}